\documentclass{amsart}
\usepackage{amsmath, amsfonts, amssymb}
\usepackage{stmaryrd}
\usepackage{mathtools}
\usepackage{amscd}
\usepackage{hhline}
\usepackage{graphicx}
\usepackage{tableau}
\usepackage{ytableau}

\usepackage[colorlinks=true, linkcolor=blue,
  citecolor=blue, urlcolor=blue]{hyperref}

\def\newthm#1#2{\newtheorem{#1}[dummy]{#2}%
  \expandafter\def\csname#2\endcsname##1{\hyperref[#1:##1]{#2~\ref*{#1:##1}}}}
\newthm{thm}{Theorem}
\newthm{lemma}{Lemma}
\newthm{prop}{Proposition}
\newthm{cor}{Corollary}
\newthm{conj}{Conjecture}
\newthm{cons}{Consequence}
\newthm{fact}{Fact}
\newthm{obs}{Observation}
\newthm{guess}{Guess}
\theoremstyle{definition}
\newthm{defn}{Definition}
\newthm{remark}{Remark}
\newthm{example}{Example}
\newthm{question}{Question}
\newthm{notation}{Notation}

\newcommand{\Section}[1]{\hyperref[sec:#1]{Section~\ref*{sec:#1}}}
\newcommand{\Table}[1]{\hyperref[tab:#1]{Table~\ref*{tab:#1}}}
\newcommand{\eqn}[1]{\hyperref[eqn:#1]{(\ref*{eqn:#1})}}
\newcommand{\Cor}[1]{\hyperref[cor:#1]{Corollary~\ref*{cor:#1}}}

\DeclareMathOperator{\Gr}{Gr}

\DeclareMathOperator{\LG}{LG}

\DeclareMathOperator{\OG}{OG}

\DeclareMathOperator{\diag}{diag}

\newcommand{\bP}{{\mathbb P}}

\newcommand{\ev}{\operatorname{ev}}

\newcommand{\wb}{\overline}

\newcommand{\pic}[2]{\includegraphics[scale=#1]{#2}}
\newcommand{\ignore}[1]{}

\newcommand{\Mb}{\wb{\mathcal M}}

\begin{document}

\title[]{A dual basis for the equivariant quantum $K$-theory of cominuscule varieties}

\author{Kevin Summers}
\email{summerskb9@gmail.com}


\date{July 31, 2024}

\begin{abstract} 
 The equivariant quantum $K$-theory ring of a flag variety is a Frobenius algebra equipped with a perfect pairing called the quantum $K$-metric. 
 It is known that in the classical $K$-theory ring for a given flag variety the ideal sheaf basis is dual to the Schubert basis with regard to the sheaf Euler characteristic. 
 We define a quantization of the ideal sheaf basis for the equivariant quantum $K$-theory of cominuscule flag varieties. 
 These quantized ideal sheaves are then dual to the Schubert basis with regard to the quantum $K$-metric. 
 We prove explicit type-uniform combinatorial formulae for the quantized ideal sheaves in terms of the Schubert basis for any cominuscule flag variety. 
 We also provide an application ultilizing the quantized ideal sheaves to calculate the Schubert structure constants associated to multiplication by the top exterior power of the tautological quotient bundle in $QK_T(Gr(k,n))$. 
\end{abstract}

\maketitle

\section{Introduction}\label{sec:intro} Let $X=G/P$ be a flag variety given by a semisimple Lie group $G$ and a parabolic subgroup $P$.
The structure sheaves of the associated Schubert varieties provide a basis for both the $K$-theory ring $K(X)$, and the (small) equivariant quantum $K$-theory ring $QK_T(X)$ defined by Givental \cite{givental:wdvv}. 
This ring $QK_T(X)$ generalizes $K(X)$ in that calculations within $QK_T(X)$ recover analogous calculations in $K(X)$ after taking appropriate restrictions. 
One such calculation is the non-degenerate bilinear pairing given by the sheaf Euler characteristic of the product of two $K$-theory classes. 
It has been proven by Brion \cite{brion:positivity} that the basis dual to the Schubert basis under this pairing will be the ideal sheaves of the boundary of Schubert varieties. 
This pairing generalizes to the quantum $K$-metric on the equivariant quantum $K$-theory ring.
It follows that there must be a quantum analogue of the ideal sheaf dual basis with respect to this new pairing. 
It is then a natural question to try to understand this dual basis and how to write its elements in terms of the Schubert classes of the space. In the cohomology of flag varieties the basis of Schubert classes is dual to the basis of opposite Schubert classes with respect to the Poincare pairing. 
This duality extends to the quantum cohomology ring and as such this problem is already completely understood in this setting. 
The ideal sheaf basis dual to the Schubert basis in the classical $K$-theory ring has been used to discuss the structure constants that appear in the multiplication of two Schubert classes, and has been used by Brion to prove the positivity result of said structure constants \cite{brion:positivity}. 
More general structure constants for other products in $K(X)$ can be found by using these dual classes. 
We give an example of such an application in the quantum $K$-theory ring using the quantum analogue of the ideal sheaf basis in \Section{application}. 
These quantized ideal sheaves have implicitly been studied in \cite{buch.mihalcea:quantum}, but only in the Type A Grassmannian, $Gr(k,n)$. These classes dual to a Schubert class in the quantum $K$-theory ring have also been discussed in \cite[Eq. 5.42]{gorbounov.korff:integrability} within the framework of integrable systems. 
The main result of this paper is a combinatorial type-uniform formula for the quantized ideal sheaf dual basis in $QK_T(X)$ when $X$ is a cominuscule flag variety. 
A Grassmann variety $Gr(k, n)$ of type A, a Lagrangian Grassmannian $LG(n, 2n)$, a maximal orthogonal Grassmannian $OG(n, 2n)$, a quadric hypersurface $Q^n$, and two exceptional varieties called the Cayley plane $E_6/P_6$ and the Freudenthal variety $E_7/P_7$ are all such cominuscule varieties. 
The Schubert classes of a cominuscule variety can be indexed by diagrams of boxes that generalize the Young diagrams known from the Schubert calculus of classical Grassmannians. 
Notably, we give a multiplication-free formula for writing the quantized ideal sheaf basis in terms of Schubert classes that is completely determined by combinatorial diagrams of boxes.

\subsection{Statement of Results} Let $X=G/P$ be a cominuscule flag variety. 
Here $G$ is a complex semisimple Lie group and $P$ is a maximal parabolic subgroup. The root system of $X$ then has a unique cominuscule root $\gamma$ and $s_\gamma\in W^P$ is the corresponding simple reflection. 
Let $T$ be the associated algebraic torus, and let $QK_T(X)$ denote the $T$-equivariant quantum $K$ ring of $X$, defined by Givental and Lee \cite{givental:wdvv, lee:quantum}.
In the classical $T$-equivariant $K$ ring, $K^T(X)$, the ideal sheaf classes, $\mathcal{I}^\lambda$, form the basis dual to the Schubert basis, $\mathcal{O}_\lambda$, with regard to the sheaf Euler characteristic.
The main results of this paper are combinatorial formulae for the quantum deformation of this dual basis in $QK_T(X)$. 

\begin{thm}\label{thm:result1} Let $\mathcal{I}_q^\mu=\alpha^{\mu}\star(1-\mathcal{O}^{s_\gamma})$ where 
$$\alpha^\mu=\sum_{\epsilon/\mu \text{ short skew shape}} \frac{\sqrt{J_\mu J\epsilon}}{J_\mu J_\epsilon}\mathcal{O^\epsilon}.$$ Then the collection $\{\mathcal{I}_q^\mu \ | \ \mu \in W^P\}$ forms the dual basis with respect to the Schubert basis, $\{\mathcal{O}_\lambda \ | \ \lambda \in W^P\}$, under the quantum $K$-metric for $QK_T(X)$ when $X$ cominuscule, i.e. ,
\[ \left(\!\left( \mathcal{O}_\lambda, \mathcal{I}_q^\mu \right)\!\right) \/=\delta_{\lambda\mu}.\]
\end{thm}

The pairing in the theorem is the quantum $K$-metric which we define in \Section{Quantum1}. 
The $\frac{\sqrt{J_\mu J\epsilon}}{J_\mu J_\epsilon}$ coefficient for the Schubert class is an element of $K^T({\text{point}})$ that restricts to $1$ in the nonequivariant setting. 
See \Section{classical results} and \Lemma{weight} for a precise definition. 
The summation for $\alpha^\mu$ is indexed by $\epsilon\in W^P$ such that $\epsilon/\mu$ is a short skew shape, meaning a skew shape composed of boxes that correspond to short roots. 
Many of the cominuscule flag varieties are also minuscule, meaning their root systems will not contain short roots. 
In these cases this summation for $\alpha^\mu$ reduces to the single Schubert class, $\mathcal{O}^\mu$. 

The proof involves first proving that a very similar formula holds in $K^T(X)$ for $\mathcal{I}^{\mu}$. 
This formula for the classical ideal sheaves is proven by manipulating the Chevalley formula for $K^T(X)$ found in \cite{buch.chaput.ea:qkchev}. 
We then show that the Gromov Witten invariants of degree $d\geq 1$ involved in calculating $ \left(\!\left(\mathcal{O}_\lambda,\mathcal{I}_q^\mu\right)\!\right)$ all vanish. 
We conclude that the formula for $\mathcal{I}_q^\mu$ in $QK_T(X)$ is the same as that for $\mathcal{I}^\mu$ in $K^T(X)$, where we replace the classical product with the quantum product. 

We also give a formula for $\mathcal{I}_q^\mu$ that is multiplication-free. 
This allows one to write the quantized ideal sheaves in terms of Schubert classes by only making use of combinatorial diagrams. 
This result makes use of the Chevalley formula for $QK_T(X)$ from \cite{buch.chaput.ea:qkchev} which relies on the first curve neighborhood of a Schubert class. 
Let $\psi:K^T(X)\to K^T(X)$ be defined by taking the first curve neighborhood, $\psi(\mathcal{O}^\mu)=\mathcal{O}^{\mu(-1)}$. 
Given $\mu\in W^P$, there is a diagrammatic procedure to find $\mu(-1)\in W^P$ in the setting of cominuscule varieties. 
\begin{thm}\label{thm:result2} We have the following equality: 
 $\mathcal{I}^\mu_q=\mathcal{I}^{\mu}-q\psi\mathcal{I}^{\mu}.$ Furthermore, if $z_1\nleqslant \mu$, then $\psi\mathcal{I}^\mu=0$. If $z_1\leq \mu$ then $\psi\mathcal{I}^{\mu}=\mathcal{I}^{\mu(-1)}.$
\end{thm}
Here $z_1$ is a particular Weyl group element that plays a role in finding the first curve neighborhood of a Schubert variety. 
The element $z_1$ can be understood as the maximum partition such that $z_1(-1)=\varnothing$. 
See \Table{tablez1}.
When $\psi\mathcal{I}^{\mu}=0$ we then have that the ideal sheaf and the quantized ideal sheaf agree. 
Because there are known formulae for $\mathcal{I}^\mu$ and its curve neighborhood that only require combinatorial diagrams, this theorem allows one to also write $\mathcal{I}_q^\mu$ through only the use of these diagrams. We give a short example of this procedure.

\begin{example}
    Let $X=Gr(2,4)$, and let $\mu=(2,1)$. 
    \begin{center}
    \ytableausetup{smalltableaux}
        \ydiagram{2,1}
    \end{center}
    Then from \cite[Lemma 3.4]{buch.chaput.ea:qkchev}, $\mathcal{I}^{(2,1)}$ is given by the rook strips $\nu_i/\mu$ for $\nu_1=(2,1)$ and $\nu_2=(2,2)$.
    \begin{center}
        \ydiagram{2,1}
        \hspace{2mm}
        \ydiagram[*(white) \bullet]
    {0,1+1}
    *[*(white)]{2,1}
    \end{center}
    The box added to $\mu$ to get $\nu_2$ is marked with a dot. 
    We have that $$\mathcal{I}^{(2,1)}=\mathcal{O}^{(2,1)}-\mathcal{O}^{(2,2)}.$$
Taking the first curve neighborhood, $\nu_1(-1)=\varnothing$ and $\nu_2(-1)=(1)$.
\begin{center}
       \ydiagram{2,1}
    *[*(gray)]{2,1}
    $\xmapsto{\psi}$ 
    $\varnothing$
    \hspace{2mm}
    \ydiagram[*(white) \bullet]
    {0,1+1}
    *[*(gray)]{2,1}
    $\xmapsto{\psi}$ 
    \ydiagram[*(white) \bullet]
    {1}
    \ytableausetup{nosmalltableaux}
    \end{center}
The boxes in gray are removed when taking the first curve neighborhood. 
Thus $$\psi\mathcal{I}^{(2,1)}=\mathcal{O}^{\varnothing}-\mathcal{O}^{(1)}.$$ 
Putting these together we have that 
$$\mathcal{I}^{(2,1)}_q=\mathcal{O}^{(2,1)}-\mathcal{O}^{(2,2)}-q(\mathcal{O}^{\varnothing}-\mathcal{O}^{(1)}).$$
In this example $\mu=(2,1)=z_1$ in $Gr(2,4)$, so we know $\psi\mathcal{I}^{\mu}$ should agree with $\mathcal{I}^{\mu(-1)}$. If we calculate $\mathcal{I}^{\varnothing}$ we then find that $\varnothing/\varnothing$ and $(1)/\varnothing$ give the only rook strips. Thus 
$$\mathcal{I}^{(2,1)(-1)}=\mathcal{O}^{\varnothing}-\mathcal{O}^{(1)}=\psi\mathcal{I}^{(2,1)}.$$
\end{example}

We go on to give an application of our dual basis in order to find particular structure constants in $QK_T(X)$. 
Let $\mathcal{Q}$ denote the tautological quotient bundle of $X=Gr(k,n)$, the type A Grassmannian. 
\begin{thm}\label{thm:result3}
We have the following equality in $QK_T(Gr(k,n))$: 
$$\det\mathcal{Q}\star \mathcal{O}^\mu=\sum_{\lambda\in W^P}(\prod_{l=1}^{n-k}T_{j_l})q^{d(\mu,\lambda)}\mathcal{O}^\lambda$$
 where the $j_l$ correspond to the last $n-k$ entries of the one line notation for the Grassmannian permutation $w_\lambda$, and $d(\mu,\lambda)=dist_X(\mu,\lambda)$ is the minimal degree $d$ for which $q^d$ occurs in the product $X^\mu \star X_\lambda$ in the quantum cohomology ring of $X$. 
\end{thm}
The coefficient of $\mathcal{O}^\lambda$ is simply $q^{d(\mu,\lambda)}$ if we restrict to the nonequivariant setting. 
See \Section{application} for details on the equivariant terms.
We calculate this product in $QK_T(X)$ using the quantum $K$ metric and thus make use of the dual basis elements, $\mathcal{I}^q_\lambda$, to simplify the calculation. 
This specific result demonstrates the more general technique that can be applied to find other structure constants in $QK_T(X)$. 
We rely on a result from \cite{gu.mihalcea.sharpe.zou:quantum} for the product $\det\mathcal{S}\star\det\mathcal{Q}$ which has only been proven for $X=Gr(k,n)$. 
Here $\mathcal{S}$ is the tautological bundle on $X$. There is a conjectured formula for the product $\det\mathcal{S}\star\det\mathcal{Q}$ for the Lagrangian Grassmannian in Lie type C which can be found in \cite{gu.mihalcea.sharpe.zou:symplectic}.

We now give a short outline of the contents of this paper. 
\Section{prelim} provides preliminary definitions and results before we prove the main result.
This includes a discussion of the classical ideal sheaves and introduces equivariant quantum $K$-theory in detail. 
\Section{classical results} gives a new formula for the ideal sheaves in $K^T(X)$ which is later used to prove a very similar formula for the quantized ideal sheaves in $QK_T(X)$. 
We provide examples of how this new formula coincides with previously known calculations for the ideal sheaves.
\Section{quantum results} contains the formulae for the quantized ideal sheaves. 
We first give a necessary lemma, then prove \Theorem{result1} utilizing the results from the prior section. 
We proceed to use \Theorem{result1} along with the Chevalley formula for $QK_T(X)$ from \cite{buch.chaput.ea:qkchev} to prove \Theorem{result2}.
This second formula for the quantized ideal sheaves is easier to calculate in practice. 
We then discuss the combinatorics involved in writing the quantized ideal sheaves in terms of Schubert classes in a manner that does not require calculating any $K$-theory products. 
Examples of how the combinatorial diagrams are used in the formulae are included. 
\Section{application}  provides a brief example of how the quantized ideal sheaves can be useful when working with Schubert classes in $QK_T(X)$. 
This example demonstrates how the quantized ideal sheaves can be used to find structure constants in a way that does not depend on any recursive calculations. 
Because the technique used is very general, the calculation of \Theorem{result3} provides an indication of the usefulness of the quantized ideal sheaves within $QK_T(X)$.

The author was supported in part by the NSF grant DMS-2152294. 
Throughout, calculations in both classical $K$-theory and quantum $K$-theory utilize the Maple based \textbf{Equivariant Schubert Calculator}, written by Anders Buch, which is available at \textbf{https://sites.math.rutgers.edu/$\thicksim$asbuch/equivcalc/}.

\section{Preliminaries}\label{sec:prelim}

\subsection{Equivariant \texorpdfstring{$K$}{Lg}-theory.}\label{sec:equivariant} We will begin by giving a brief overview of equivariant $K$-theory, as this ring and its quantum generalization will be the primary settings in which we will be working.
Let $X$ be a complex algebraic variety with an action of an algebraic torus $T=(\mathbb{C}^{*})^n$. 
The Grothendieck group of $T$-equivariant coherent $\mathcal{O}_X$-modules, $K_T(X)$, is a module over $K^T(X)$, the ring of equivariant vector bundles over $X$. 
However, when $X$ is non-singular, the map $K^T(X)\to K_T(X)$ that sends a vector bundle to its sheaf of sections is an isomorphism. 
Since we will take $X$ to be a cominuscule flag variety, $X$ will be non-singular. 
Let $\Gamma=K^T(\text{point})$, then the pullback along the structure morphism $X\to \{\text{point}\}$ induces a $\Gamma$-algebra structure on $K^T(X)$.
This ring $\Gamma$ can be understood as the virtual representations of $T$, with $\mathbb{Z}$-basis consisting of the classes $[\mathbb{C}_\alpha]$.
Here $\alpha:T\to \mathbb{C}^{*}$ is a character and $\mathbb{C}_\alpha$ is the one-dimensional representation of $T$ defined by $t.z=\alpha (t)z$ for $t\in T$ and $z\in \mathbb{C}_\alpha$.
Furthermore, for our purposes $X$ will always be a projective variety, and thus we have a pushforward along the structure morphism of $X$, $\chi_X:K_T(X)\to \Gamma$. The map $\chi_X$ is a homomorphism of $\Gamma$-modules by the projection formula. 
We use the same notation to denote the sheaf Euler characteristic,  $\chi_X: K^T(X) \times K^T(X)\to \Gamma$ which is a symmetric nondegenerate bilinear form on $K^T(X)$. 
When we say that the ideal sheaf basis for $K^T(X)$ is dual to the basis of Schubert classes, this duality is with respect to the sheaf Euler Characteristic. 
For a more complete reference on equivariant $K$-theory see \cite{chriss.ginzburg:representation, fulton:intersection}.

\subsection{Schubert Varieties}\label{sec:Schubert1} As flag varieties, cominuscule varieties are of the form $G/P$ under the usual setting $T\subset B \subset P\subset G$. As such they have a stratification into Schubert cells.
Here $G$ is a complex semisimple Lie group, $P$ is a parabolic subgroup, $B$ is a Borel subgroup, and $T$ is a maximal torus.
We let $W=N_G(T)/T$ be the Weyl group of $G$, let $W_P=N_P(T)/T$ be the Weyl group of $P$, and let $W^P\subset W$ be the set of minimal representatives of the cosets in $W/W_P$.
When $G$ has Lie type A, $W=S_n$, and $W^P$ is isomorphic to the group of Grassmannian permutations.
Each element $w\in W$ defines a $B$-stable Schubert variety $X_w=\overline{Bw.P}$ and an opposite $B^{-}$-stable Schubert variety $X^w=\overline{B^{-}w.P}$ in $X=G/P$.
Here $B^{-}\subset G$ is the opposite Borel subgroup such that $B\cap B^{-}=T$. 
If we take $w\in W^P$ so that it is a minimal representative, then we have $dim(X_w)=codim(X^w,X)=l(w)$ where $l(w)$ denotes the length of $w$. 
Any $u\in W^P$ has a dual element, denoted $u^{\vee}=w_0uw_P$, where $w_0$ is the longest element in $W$ and $w_P$ is the longest element in $W_P$. 
The diagrammatic combinatorics of cominuscule varieties requires discussion of their root systems.
We let $\Phi$ be the root system of $(G,T)$, interpreted as a set of characters of $T$, let $\Phi^{+}$ be the set of positive roots determined by $B$, let $\Delta\subset\Phi^{+}$ be the simple roots, and let $\Delta_P\subset \Delta$ be the set of simple roots $\beta$ where the associated reflection $s_\beta$ is in $W_P$.

\subsection{Cominuscule Varieties}\label{sec:Cominuscule1} A simple root $\gamma\in \Delta$ is called \textit{cominuscule} if the coefficient of $\gamma$ is one when the highest root is written as a linear combination of simple roots. 
The flag variety $X=G/P$ is cominuscule if $\Delta\setminus\Delta_P$ consists of a single cominuscule root $\gamma$. Furthermore, $X$ is \textit{minuscule} if the root system $\Phi$ is simply laced, meaning the roots are all of the same length. 
We take roots within a simply laced root system to be considered long roots. 
The cominuscule flag varieties that are not minuscule are the Lagrangiann Grassmannians and the quadrics of odd dimension. 
We will see that short roots play a role in the combinatorics of both the ideal sheaves and the quantized ideal sheaves. Because minuscule flag varieties have no short roots, the combinatorics of their ideal sheaves will simplify significantly. 
In particular, many sums will reduce to a single term in the minuscule setting.

One of the great benefits of working in the cominuscule setting is the use of combinatorial diagrams to discuss Schubert calculus calculations. We will now outline how these diagrams come about.
Assume $X$ is cominuscule. From a result by Proctor we then have that the Bruhat order on $W^P$ is a distributive lattice that agrees with the left weak Bruhat order \cite{proctor:bruhat,stembridge:fully}. 
There is a partial order on the root lattice $\text{Span}_\mathbb{Z}(\Delta)$ defined by $\alpha'\leq \alpha$ if and only if $\alpha - \alpha'$ can be written as a sum of positive roots. 
Let $\mathcal{P}_X=\{\alpha\in \Phi|\alpha\geq \gamma\}$, with the induced partial order from the root lattice. 
Then, since $X$ is cominuscule, this is also the set of positive roots $\alpha$ for which the coefficient of the cominuscule root $\gamma$ is one. 
For any element $u\in W$ we let $I(u)=\{\alpha\in \Phi^{+} | u.\alpha<0\}$ denote the inversion set of $u$ and so $l(u)=|I(u)|$. If we restrict to only $u\in W^P$ then the map given by $u\mapsto I(u)$ is a bijection between the elements of $W^P$ and the lower order ideals of $\mathcal{P}_X$. 
This assignment is order preserving in the sense that $u\leq v$ if and only if $I(u)\subset I(v)$. 
Given a lower order ideal $\lambda\subset \mathcal{P}_X$ let $\lambda=\{\alpha_1,\alpha_2,\ldots,\alpha_{|\lambda|}\}$ be any ordering of its elements that is compatible with the partial order $\leq$, meaning $\alpha_i<\alpha_j$ implies $i<j$. 
Then the element of $W^P$ corresponding to $\lambda$ is the product of reflections $w_\lambda=s_{\alpha_1} s_{\alpha_2} \ldots s_{\alpha_{|\lambda|}}$. 
When $X$ is of Lie Type A, the set of order ideals $\lambda$ is in bijective correspondence with the set of Young diagrams used to discuss the combinatorics of Schubert calculus with regards to the Grassmannian, $Gr(k,n)$. 
Similarly when $X$ is of Lie Type C, the order ideals correspond to the shifted Young Diagrams that are used within the setting of the Lagrangian Grassmannian.
The bijection between order ideals of $\mathcal{P}_X$ and elements of $W^P$ allows us to generalize these combinatorial diagrams to any cominuscule $X$. 
For this reason the roots in $\mathcal{P}_X$ will sometimes be called $\textit{boxes}$, and the order ideal $I(u)$ will be called the \textit{shape} of $u$.
Since the set of $u\in W^P$ is in one-to-one correspondence with the set of $I(u)$ and the diagrams of these shapes will play a role in combinatorial formulae we will often abuse notation and write $u$ when referring to its shape.
Given a box, $\alpha$, we will sometimes write $u\setminus \alpha$ to mean the shape of $u$ with the box $\alpha$ removed. Similarly we will write $u\cup \alpha$ to mean the shape of $u$ with the box $\alpha$ added. In order to use these operations one must check that they are are well defined, meaning that the result is a shape that corresponds to an element in $W^P$.

Given two elements $u,w$ in $W^P$ such that $u\leq w$ we will write $w/u=wu^{-1}\in W$. 
The Bruhat order on $W^P$ agrees with the left weak Bruhat order, so we have that $l(w/u)=l(w)-l(u)$. 
For any $\alpha\in \mathcal{P}_X$, fix the order ideal $\lambda=\{\alpha'\in \mathcal{P}_X|\alpha'<\alpha\}$ consisting of the roots smaller than $\alpha$, and the define $\delta(\alpha)=w_\lambda.\alpha$. 
Then $s_{\delta(\alpha)}=w_{\lambda}s_\alpha w_\lambda^{-1}=w_{\lambda\cup\alpha}/w_\lambda$ has length one and is a simple reflection. 
We then have that $\delta:\mathcal{P}_X\to\Delta$ is a labeling of the boxes in $\mathcal{P}_X$ by simple roots. 
Furthermore, if $I(w)\setminus I(u)=\{\alpha_1,\alpha_2,\ldots,\alpha_r\}$ is any ordering compatible with $\leq$, then $w/u=s_{\delta(\alpha_r)}\cdots s_{\delta(\alpha_2)}s_{\delta(\alpha_1)}$ is a reduced expression for $w/u$, in that $l(w/u)=\sum_i l(s_{\delta(\alpha_i)})=r$. 
Thus we have that the \textit{skew shape}, $I(w)\setminus I(u)$, completely determines $w/u$. 
If all the simple roots $\delta(\alpha_r),\ldots,\delta(\alpha_1)$ in the reduced expression for $w/u$ are short roots, then we call $I(w)\setminus I(u)$ a \textit{short skew shape}. 
We will abuse notation and call $w/u$ a (short) skew shape as well. The diagram for $w/u$ can be understood as the skew diagram that keeps all the boxes from the diagram of $w$ that remain after removing the boxes that are in the diagram of $u$.

We call $w/u$ a \textit{rook strip} if, as element of $W$, $w/u$ is a product of commuting simple reflections. 
Equivalently, no pair of roots in $I(w)\setminus I(u)$ are comparable by the order $\leq$. 
We include $u/u$, the empty skew shape, as a rook strip. 
In terms of diagrams, two boxes in a diagram will be comparable if they are in the same column or in the same row. 
We say $w/u$ is a \textit{short rook strip} if it is a product of commuting reflections corresponding to short simple roots. 
In other words $I(w)\setminus I(u)$ is made of incomparable short roots. 
Notice that if $X$ is minuscule, then the root system $\Phi$ is simply laced, all roots are long by convention, and $w/u$ is a short rook strip if and only if $w=u$.

\subsection{Schubert Structure Constants}\label{sec:Schubert2} The Schubert structure constants provide insight into the role that the ideal sheaves play in the Schubert calculus of $K^T(X)$.
The equivariant $K$-theory ring $K^T(X)$ of the flag variety $X$ has a basis over $\Gamma$, indexed by $w\in W^P$, consisting of the (opposite) \textit{Schubert classes} $\mathcal{O}^w=[\mathcal{O}_{X^w}]$. 
The classes $\mathcal{O}_w=[\mathcal{O}_{X_w}]$ also form a basis. 
The \textit{Schubert structure constants} of $K^T(X)$ are the classes $N^{w,0}_{u,v}\in \Gamma$ defined for $u,v,w\in W^P$ by the identity 
$$\mathcal{O}^u\cdot\mathcal{O}^v=\sum_w N^{w,0}_{u,v}\mathcal{O}^w.$$

Define $\mathcal{O}^{\vee}_w\in K^T(X)$ to be the basis element dual to $\mathcal{O}^w$, given by $\chi_X(\mathcal{O}^u\cdot \mathcal{O}^{\vee}_w)=\delta_{u,w}$ for $u,w\in W^P$. 
It follows that $N^{w,0}_{u,v}=\chi_X(\mathcal{O}^u\cdot \mathcal{O}^v\cdot \mathcal{O}^{\vee}_w)$. The \textit{boundary} of the Schubert variety $X_w$ is the closed subvariety defined by $\partial X_w=X_w\setminus Bw.P.$ 
It was proven by Brion that 
$$\mathcal{O}^{\vee}_w=[I_{\partial X_w}]$$
where $I_{\partial X_w}\subset \mathcal{O}_{X_w}$ denotes the ideal sheaf of this boundary \cite{brion:positivity}. 
As notation we will take $\mathcal{I}_w=[I_{\partial X_w}]$ to be this element dual to $\mathcal{O}^w$ under the pairing $\chi_X$.
Similarly we will take $\mathcal{I}^w$ to be dual to $\mathcal{O}_w$ under $\chi_X$. 
In \Section{classical results} we will primarily focus on $\mathcal{I}^w$ so that we may more easily write the ideal sheaves in terms of the (opposite) Schubert classes, $\mathcal{O}^u$, that are indexed by codimension.  
We discuss how to find $\mathcal{I}_w$ from $\mathcal{I}^w$ after \Theorem{classical}. 

Working in the equvariant K-theory ring, $K^T(X)$ as opposed to working in the ordinary K-theory ring, $K(X)$ has some advantages. 
Let $X^T=\{w.P|w\in W^P\}$ denote the set of $T$-fixed points in $X$. 
Then the restriction $K^T(X)\to K^T(X^T)=\prod_{w\in W^P}\Gamma$ is an injective ring homomorphism, and calculations in $K^T(X)$ may be carried out in $K^T(X^T)$ instead \cite{kostant.kumar:t-equivariant, goresky.kottwitz.ea:equivariant}.
The images of Schubert classes under this map are given by the restriction formulae in \cite{andersen.jantzen.ea:representations, billey:kostant,
  graham:equivariant, willems:k-theorie} (see also
\cite{knutson:schubert}). 
Furthermore, the ring structure of $K^T(X)$ is also determined by the Chevalley formula of Lenart and Postnikov \cite{lenart.postnikov:affine}, which provides an explicit expression for the product of any Schubert class with a divisor in $K^T(X)$. 
This is not the case in the ordinary $K$-theory ring $K(X)$. 
Similarly, the ring structure of $QK_T(X)$ is completely determined by the Chevalley formula of \cite{buch.chaput.ea:qkchev}, which is not true for the ordinary quantum $K$-theory ring, $QK(X)$. 
We make use of various forms of the Chevalley formula in both $K^T(X)$ and $QK_T(X)$ throughout the subsequent sections. 
For many more useful formulae involving Schubert structure constants in various special cases one can look in \cite{buch.samuel:k-theory, knutson:puzzles, buch:mutations, pechenik.yong:equivariant}.  

\subsection{Quantum \texorpdfstring{$K$}{Lg}-theory}\label{sec:Quantum1}The equivariant quantum $K$-theory ring, $QK_T(X)$, generalizes the equivariant $K$-theory ring, $K^T(X)$, in the sense that calculations in $K^T(X)$ can first be done in $QK_T(X)$ where we then restrict the quantum parameters to $0$. 
In this way, the quantized ideal sheaves, 
$\mathcal{I}^\mu_q$, will restrict to the classical ideal sheaves, $\mathcal{I}^\mu$. 
This can be seen in \Theorem{Thm3}. 

We follow the definition for $QK_T(X)$ as outlined by Givental \cite{givental:wdvv} and Lee \cite{lee:quantum}. 
We call a homology class $d=\sum d_\beta[X_{s_\beta}]\in H_2(X;\mathbb{Z})$, indexed by $\beta\in \Delta\setminus\Delta_P$, an \textit{effective degree} if $d_\beta\geq0$ for each $\beta$.
When $d,e\in H_2(X;\mathbb{Z})$ we write $e\leq d$ if and only if $d-e$ is effective. 
For any effective degree $d$ we let $\overline M_{0,n}(X,d)$ denote the Kontsevich moduli space of $n$-pointed stable maps to $X$ of genus zero and degree $d$ \cite{fulton.pandharipande:notes}.
There are natural evaulation maps $ev_i:\overline M_{0,n}(X,d) \to X$ for $1\leq i\leq n$. 
Given an effective degree $d$, and classes $\kappa_1,\kappa_2,\ldots,\kappa_n\in K^T(X)$, we define the corresponding (equivariant $K$-theoretic) Gromov-Witten invariant of $X$ by 
$$I_d(\kappa_1,\kappa_2,\ldots,\kappa_n)=\chi_{\overline M_{0,n}(X,d)}(ev^{*}_1(\kappa_1)
\cdot ev^{*}_2(\kappa_2)\cdot\ldots\cdot ev^{*}_n(\kappa_n))\in \Gamma$$
This invariant is $\Gamma$-linear in each argument $\kappa_i$. 
It is worth noting that while Gromov-Witten invariants are necessary to discuss quantum $K$-theory, they will only briefly come up in the following proofs. 
\Lemma{GWi} shows that the necessary Gromov-Witten invariants will end up being $0$ for $d\geq1$. 
The remaining degree $0$ Gromov-Witten invariants are then equivalent to the sheaf Euler characteristic from classical $K$-theory.

The (small) $T$-equivariant quantum $K$-theory ring of $X$, denoted by $QK_T(X)$, is an algebra over the ring of formal power series $\Gamma[[q]]=\Gamma[[q_\beta : \beta\in \Delta\setminus \Delta_P]]$, which has one variable $q_\beta$ for each simple root $\beta\in \Delta\setminus\Delta_P$.
For $d=\sum_\beta d_\beta [X_{s_\beta}]$ we write $q^d=\prod_{\beta}q_\beta^{d_\beta}$. 
In the context of this paper we will take $X$ to be cominuscule, so $\Delta\setminus \Delta_P =\{\gamma\}$. 
Thus the equivariant quantum $K$-theory ring $QK_T(X)$ is an algebra over the power series ring $\Gamma[[q]]$ in a single variable $q=q_\gamma$. 
We drop the subscript and denote this variable by $q$. 
Similarly, effective degrees $d=\sum_\beta d_\beta [X_{s_\beta}]$ are only indexed by $d_\gamma$ and can thus be identified with the nonnegative integers. 
We have $QK_T(X)=K^T(X)\otimes_{\Gamma}\Gamma[[q]]$ as a module over $\Gamma[[q]]$. 
Thus $QK_T(X)$ is a free module over $\Gamma[[q]]$ with basis consisting of the Schubert classes $\mathcal{O}^w$ for $w\in W^P$.

The \textit{quantum K-metric} is the $\Gamma[[q]]$-bilinear pairing $QK_T (X) \times QK_T (X) \rightarrow \Gamma[[q]]$ determined by
$$\left(\!\left(\kappa_1,\kappa_2\right)\!\right)=\sum_{d\geq0} q^dI_d(\kappa_1,\kappa_2)$$
for $\kappa_1 , \kappa_2 \in K_T (X)$. 
Since this pairing is nondegerate, there is a unique $\Gamma[[q]]$-bilinear product $QK_T (X) \times QK_T (X) \rightarrow QK_T (X)$ defined by
$$\left(\!\left(\kappa_1\star\kappa_2,\kappa_3\right)\!\right)=\sum_{d\geq0} q^d I_d( \kappa_1,\kappa_2,\kappa_3)$$
for all 
$\kappa_1 , \kappa_2,\kappa_3 \in K_T (X)$.
The Frobenius property $\left(\!\left(\kappa_1\star\kappa_2,\kappa_3\right)\!\right)=\left(\!\left(\kappa_1,\kappa_2\star\kappa_3\right)\!\right)$ holds for all $\kappa_1 , \kappa_2,\kappa_3 \in QK_T (X)$, due to the symmetry of the Gromov-Witten invariants. 
The string identity $I_d(\kappa_1,\ldots,\kappa_n,1)=I_d(\kappa_1,\ldots,\kappa_n)$ then implies that $1\in K_T(X)$ is a multiplicative unit in $QK_T(X)$ \cite{buch.chaput.ea:euler}. 
A theorem of Givental states that this defines an associative product \cite{givental:wdvv}. 
It is with regards to the quantum $K$-metric that the quantized ideal sheaves, $\mathcal{I}_q^\mu$, will be the basis dual to the Schubert basis. 
It is worth mentioning that it is known how to calculate the quantum $K$-metric for any flag variety, $G/ P$. 
For example when $X=G/P$ is cominuscule $$\left(\!\left(\mathcal{O}^u,\mathcal{O}_v\right)\!\right)=\frac{q^{d(u,v)}}{1-q}$$ where $d(u,v)=dist_X(u,v)$ is the minimal degree $d$ for which $q^d$ occurs in the product $X^u\star X_v$ in the quantum cohomology ring, $QH(X)$. 
The pairing $\left(\!\left(\mathcal{O}^u,\mathcal{O}_v\right)\!\right)$ was first calculated in \cite{buch.mihalcea:quantum} by making use of curve neighborhoods, and \cite{buch.chaput.ea:euler} then identified the resulting first curve neighborhood as $\frac{q^{d(u,v)}}{1-q}$.

Multiplication in $QK_T(X)$ can also be defined in terms of structure constants by $$\mathcal{O}^u\star\mathcal{O}^v=\sum_{w,d\geq0}N^{w,d}_{u,v}q^d \mathcal{O}^w$$
where the sum is over all effective degrees $d$ and $w\in W^P.$ 
The structure constants $N^{w,d}_{u,v}$ can be calculated recursively by 
$$N^{w,d}_{u,v}=I_d(\mathcal{O}^u,\mathcal{O}^v,\mathcal{O}^{\vee}_w)-\sum_{\lambda,0<e\leq d}N^{\lambda,d-e}_{u,v}I_e(\mathcal{O}^\lambda,\mathcal{O}^{\vee}_w)$$
where this sum is over all $\lambda\in W^P$ and degrees $e$ for which $0<e\leq d$. 
Here note that $\mathcal{O}^{\vee}_w=\mathcal{I}_w$, the ideal sheaves in $K^T(X)$. 
However, we can instead write 
$$N^{w,d}_{u,v}=\left(\!\left(\mathcal{O}^u\star\mathcal{O}^v,\mathcal{I}_w^q\right)\!\right)$$
where $\mathcal{I}^q_w$ is the element dual to $\mathcal{O}^w$ with regard to the quantum K-metric in $QK_T(X)$. 
In this sense the quantized ideal sheaves allow us to write the Schubert structure constants of $QK_T(X)$ without a recursive calculation. 

\subsection{Curve Neighborhoods}\label{sec:curve}
The Chevalley formula in both $K^T(X)$ and $QK_T(X)$ can be formulated combinatorially in terms of adding or removing boxes from diagrams that represent Weyl group elements. 
The same can be said for formulae for writing ideals sheaf classes in terms of Schubert classes. 
In order to describe the quantized ideal sheaf classes, $\mathcal{I}_q^\mu$ in $QK_T(X)$ in terms of adding or removing boxes from diagrams, we need to discuss curve neighborhoods. 
Specifically, if we are given $\mu\in W^P$ we only require the first curve neighborhood of $\mu$, which we will call $\mu(-1)$. 
For a more robust discussion on curve neighborhoods see \cite{buch.chaput.ea:qkchev, buch.mihalcea:curve}. 

Let $\Omega\subset X$ be any closed subvariety and let $d\in H_2(X;\mathbb{Z})$ be an effective degree, then we let $\Gamma_d(\Omega)$ denote the closure of the union of all rational curves of degree $d$ in $X$ that pass through $\Omega$. 
Equivalently, $\Gamma_d(\Omega)=ev_2(ev_1^{-1}(\Omega))$  where $ev_1,ev_2:\overline M_{0,2}(X,d)\to X$ are the evaluation maps. 
Furthermore $\Gamma_d(\Omega)$ is irreducible whenever $\Omega$ is irreducible \cite{buch.chaput.ea:finiteness}.
If $\Omega$ is then a $B$-stable Schubert variety in $X$, we have that  $\Gamma_d(\Omega)$ is as well. 
Given $w\in W^P$ we then have well-defined $w(d)$ and $w(-d)$ in $W^P$ given by the identities $\Gamma_d(X_w)=X_{w(d)}$ and $\Gamma_d(X^w)=X^{w(-d)}$. 
When $X$ is a cominuscule variety, the Schubert class associated to $w(-1)\in W^P$ will be $\mathcal{O}^{w(-1)}$, the class of the structure sheaf of the Schubert variety $X^{w(-1)}=\Gamma_1(X^w)$. 
For $d\geq 0$ we define $z_d\in W^P$ to be the unique element such that $X_{z_d}=\Gamma_d(1.P)$. 
This element satisfies that $z_d w_P$ is inverse to itself, where $w_P$ denotes the longest element in $W_P$.

The following definition only comes up in the proof of \Lemma{GWi}. 
Let $\Gamma_d(X_u, X^v)$ denote the subvariety of $X$ given by the union of all stable curves of degree $d$ that pass through $X_u$ and $X^v$.
Equivalently we have
$\Gamma_d(X_u,X^v) = \ev_3(\ev_1^{-1}(X_u) \cap \ev_2^{-1}(X^v))$,
where $\ev_1, \ev_2, \ev_3 : \Mb_{0,3}(X,d) \to X$ are the evaluation maps.

\def\vmm#1{\vspace{#1mm}}
\begin{table}
\caption{Partially ordered sets of cominuscule varieties
  with $I(z_1)$ highlighted.}
\label{tab:tablez1}
\begin{tabular}{|c|c|}
\hline%%%%%%%%%%%%%%%%%%%%%%%%%%%%%%%%%%%%%%%%%%%%%%%%%%%%%%%%%%%%%%%%%%%%%
&\vmm{-2}\\
Grassmannian $\Gr(3,7)$ of type A & Max.\ orthog.\ Grassmannian $\OG(6,12)$
\\
&\vmm{-3}\\
\pic{1}{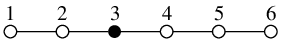} &\\
&\\
&\vmm{-3}\\
%$\tableau{12}{
%[aLlTt]3 & [aTtBb]4 & 5 & [aTtBbRr]6 \\
%[aLlRr]2 & [a]3 & 4 & 5 \\
%[aLlRrBb]1 & [a]2 & 3 & 4
%}$
\ydiagram{4,4,4}
    *[*(gray)]{4,1,1}
&\vmm{-27}\\
&\pic{1}{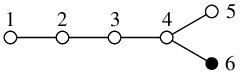}
\\ &
%$\tableau{12}{
%[aLlTtBb]6 & [aTt]4 & 3 & 2 & [aTtRr]1 \\
%  & [aLlBb]5 & [aBb]4 & 3 & [aBbRr]2 \\
%  &   & [a]6 & 4 & 3 \\
%  &   &   & 5 & 4 \\
%  &   &   &   & 6
%}$
    \ydiagram{0+0,0+0,2+3,3+2,4+1}
    *[*(gray)]{5,1+4}
\\ &
\vmm{-2}\\
\hline%%%%%%%%%%%%%%%%%%%%%%%%%%%%%%%%%%%%%%%%%%%%%%%%%%%%%%%%%%%%%%%%%%%%%
&\vmm{-2}\\
Lagrangian Grassmannian $\LG(6,12)$ & Cayley Plane $E_6/P_6$
\\
&\vmm{-3}\\
\pic{1}{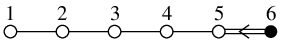} & \\
&\vmm{-2}\\
%$\tableau{12}{
%[aLlTtBb]6 & [aTtBb]5 & 4 & 3 & 2 & %[aTtBbRr]1 \\
%  & [a]6 & 5 & 4 & 3 & 2 \\
%  &   & 6 & 5 & 4 & 3 \\
%  &   &   & 6 & 5 & 4 \\
%  &   &   &   & 6 & 5 \\
%  &   &   &   &   & 6
%}$
    %\ydiagram[*(white)\bullet]
    %{3+0,3+0,2+1}
    %{0+0,1+5,2+4,3+3,4+2,5+1}
    %*[*(gray)]{6,0}
        \ydiagram[*(white)\bullet]
    {0,1+1,2+1,3+1,4+1,5+1}
    *[*(white)]{0+0,2+4,3+3,4+2,5+1}
    *[*(gray)\bullet]{1,0,0,0,0,0}
    *[*(gray)]{1+5,0,0,0,0,0}
&\vmm{-36}\\
& \pic{1}{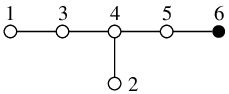} \\
&\vmm{-2}\\
&
%$\tableau{12}{
%[aLlTtBb]6 & [aTtBb]5 & [aTt]4 & 3 & %[aTtRr]1 \\
%  &   & [aLlBb]2 & [a]4 & [aRr]3 \\
%  &   &   & [aLlBb]5 & [aBb]4 & [aTtBbRr]2 \\
%  &   &   & [a]6 & 5 & 4 & 3 & 1
%}$
    \ydiagram{0+0,0+0,0+0,3+5}
    *[*(gray)]{5,2+3,3+3}
\\
& \vmm{-2}\\
\hline%%%%%%%%%%%%%%%%%%%%%%%%%%%%%%%%%%%%%%%%%%%%%%%%%%%%%%%%%%%%%%%%%%%%%
& \vmm{-2}\\
Even quadric $Q^{10} \subset \bP^{11}$ & Freudenthal variety $E_7/P_7$
\\
&\vmm{-3}\\
\pic{1}{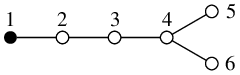} &
\\
&\vmm{-3}\\
%$\tableau{12}{
%[aLlTtBb]1 & [aTtBb]2 & 3 & [aTt]4 & %[aTtRr]5 \\
%  &   &   & [aLlBb]6 & [aBb]4 & [aTtBb]3 & [aTtBbRr]2 & [a]1
%}$
    \ydiagram{0+0,7+1}
    *[*(gray)]{5,3+4}
& \\
&\\
\hhline{-~}%%%%%%%%%%%%%%%%%%%%%%%%%%%%%
&\vmm{-2}\\
Odd quadric $Q^{11} \subset \bP^{12}$ &
\\
&\vmm{-3}\\
\pic{1}{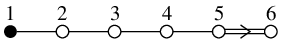} & \\
&\vmm{-1}\\
%$\tableau{12}{
%[aLlTtBb]1 & [aTtBb]2 & 3 & 4 & 5 & 6 & 5 & 4 & 3 & [aTtBbRr]2 & [a]1
%}$
        \ydiagram[*(white)\bullet]
    {10+1}
    *[*(white)]{0}
    *[*(gray)\bullet]{1}
    *[*(gray)]{1+9}
& \vmm{-52}\\
& \pic{1}{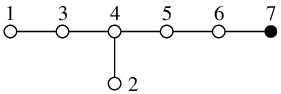} \\
&\vmm{-3}\\
&
%$\tableau{12}{
%[aLlTtBb]7 & [aTtBb]6 & 5 & [aTt]4 & 3 & %[aTtRr]1 \\
%  &   &   & [aLlBb]2 & [a]4 & [aRr]3 \\
%  &   &   &   & [aLl]5 & [a]4 & [aTtRr]2 \\
%  &   &   &   & [aLlBb]6 & [aBb]5 & 4 & [aTtBb]3 & [aTtBbRr]1 \\
%  &   &   &   & [a]7 & 6 & 5 & 4 & 3 \\
%  &   &   &   &   &   &   & 2 & 4 \\
%  &   &   &   &   &   &   &   & 5 \\
%  &   &   &   &   &   &   &   & 6 \\
%  &   &   &   &   &   &   &   & 7
%}$
    \ydiagram{0+0,0+0,0+0,0+0,4+5,7+2,8+1,8+1,8+1}
    *[*(gray)]{6,3+3,4+3,4+5}

\vmm{-2}\\
& \\
\hline%%%%%%%%%%%%%%%%%%%%%%%%%%%%%%%%%%%%%%%%%%%%%%%%%%%%%%%%%%%%%%%%%%%%%
\end{tabular}
\end{table}

\Table{tablez1} displays one cominuscule variety $X$ from each family as well as the associated Dynkin diagram and the partially ordered set $\mathcal{P}_X$. 
The marked node in the Dynkin diagram corresponds to the cominuscule simple root $\gamma$.
The roots of $\mathcal{P}_X$ are represented as boxes, and the partial order is given by $\alpha' \leq \alpha$ if and only if $\alpha'$ is located north-west of $\alpha$.
In the cases where $X$ is cominuscule but not minuscule, the boxes that correspond to long roots are marked with a dot. 
In the diagrams with no dots, all the roots are long. 
In addition, the shape $I(z_1)$ is marked in gray. 
We have that \cite[Lemma 3.1]{buch.chaput.ea:qkchev} implies, for any element $u \in W^P$, the shape of $u(-1)$ is obtained from the shape of $u$ by first removing any boxes contained in $I(z_1)$, and then moving the remaining boxes to the upper-left corner of $\mathcal{P}_X$.

\section{Ideal Sheaf Basis}\label{sec:classical results}
We begin by giving a lemma necessary to manipulate the equivariant weights that appear in both the classical and quantized ideal sheaves. 
We adapt the notation used in the Chevalley formula for $QK_T(X)$ from  \cite{buch.chaput.ea:qkchev}. Let $J_u=[\mathbb{C}_{u.w_{\gamma}-w_{\gamma}}]$ and let $\sqrt{J_v J_w}=[\mathbb{C}_{v.w_{\gamma}-w_{\gamma}-\delta(w/v)}]$ with $w/v$ a short rook strip and $$\delta(w/v)=\sum_{\alpha\in I(w)\setminus I(v)}\delta(\alpha).$$
If we extend this definition so that 
$$\delta(w/v)=\sum_{\alpha\in I(w)\setminus I(v)}\delta(\alpha)$$ 
where $w/v$ is any skew shape, then we can write $\sqrt{J_v J_w}=[\mathbb{C}_{v.w_{\gamma}-w_{\gamma}-\delta(w/v)}]$ for any skew shape $w/v$. 
We then have the following lemma.

\begin{lemma}\label{lemma:weight} For $u,v,w\in W^P$ where $u \leq v \leq w$  we have that 
$$\frac{\sqrt{J_u J_v}\sqrt{J_v J_w}}{J_v}=\sqrt{J_u J_w}.$$
\end{lemma}

\begin{proof}
    $$\sqrt{J_u J_v}\sqrt{J_v J_w}=[\mathbb{C}_{u.w_{\gamma}-w_{\gamma}+v.w_{\gamma}-w_{\gamma}-\delta(w/v)-\delta(v/u)}]$$
Note that 
$$-\delta(w/v)-\delta(v/u)=-\left(\sum_{\alpha\in I(w)\setminus I(v)} \delta(\alpha)+ \sum_{\beta\in I(v)\setminus I(u)}\delta(\beta)\right)$$
$$=-\sum_{\alpha\in I(w)\setminus I(u)} \delta(\alpha)=-\delta(w/u).$$
Thus 
$$\sqrt{J_u J_v}\sqrt{J_v J_w}=[\mathbb{C}_{u.w_{\gamma}-w_{\gamma}+(v.w_{\gamma}-w_{\gamma})-\delta(w/u)}]$$ 
and so 
$$\frac{\sqrt{J_u J_v}\sqrt{J_v J_w}}{J_v}=[\mathbb{C}_{u.w_{\gamma}-w_{\gamma}-\delta(w/u)}]=\sqrt{J_u J_w}.$$
\end{proof}

Next we give a new, alternative formula for writing an ideal sheaf in terms of multiplication of a sum of Schubert classes by a specific line bundle in $K^T(X)$.
This line bundle is denoted $J=1-\mathcal{O}^{s_\gamma}$ in \cite{buch.chaput.ea:qkchev}. 
It is used to express the Chevalley formulae of both $K^T(X)$ and $QK_T(X)$, because it allows for Gromov Witten invariants to vanish in a convenient way. 
This can be seen in \Lemma{GWi}, which is essentially the same argument used to prove the quantum Chevalley formula in \cite{buch.chaput.ea:qkchev}. 
Thus, the formula for the classical ideal sheaves in the next theorem will be useful in proving a similar formula for the quantized ideal sheaves in $QK_T(X)$. 
We will need the particular version of the Chevalley formula for $K^T(X)$, that 
\begin{equation}\label{chevI}
\mathcal{O}^\mu\cdot (1-\mathcal{O}^{s_\gamma})=\sum_{\nu/\mu \text{ short rook strip}} (-1)^{l(\nu/\mu)}\sqrt{J_\mu J_\nu}\mathcal{I}^\nu
\end{equation}
\cite[Thm. 3.6]{buch.chaput.ea:qkchev}. 
This will be used in the proof of the next theorem. 
Note that in the case that $X$ is a minuscule variety, for a given $\nu\in W^P$ the only short rook strip $\nu/\mu$ is when $\mu=\nu$, so $l(\nu/\mu)=0$ and the formulae will simplify significantly.

\begin{thm}\label{thm:classical}
Let $\mu\in W^P$. Then $\mathcal{I^\mu}=\alpha^\mu \cdot(1-\mathcal{O}^{s_\gamma})$ with 
$\alpha^\mu=\sum_{\epsilon} c^\epsilon\mathcal{O}^\epsilon$
and 
$$c^\epsilon=\left\{
    \begin{array}{lr}
        \frac{\sqrt{J_\mu J_\epsilon}}{J_\mu J_\epsilon}, & \text{if } \epsilon/\mu \text{ is a short skew shape or }\mu=\epsilon  \\
        0, & \text{else.}
    \end{array}
\right.$$
\\
\end{thm}
In the minuscule cases since there are no short roots, and the only valid $\epsilon \in W^P$ is $\epsilon=\mu$, so 
$\alpha^\mu=\frac{1}{J_\mu}\mathcal{O}^{\mu}$. 
Also the map that restricts $c^\epsilon \mapsto 1$ will give the ideal sheaves for the non-equivariant $K$-theory ring, $K(X)$.

\begin{proof} We know that $(\mathcal{O}_\lambda,\mathcal{I}^\mu)=\delta_{\lambda\mu}$, and we need to show  
that 
$$(\mathcal{O}_\lambda,\alpha^\mu\cdot(1-\mathcal{O}^{s_\gamma}))=\delta_{\lambda\mu}$$ 
by the non-degeneracy of the pairing. By \eqref{chevI} we have: 

$$(\mathcal{O}_\lambda,\alpha^\mu\cdot (1-\mathcal{O}^{s_\gamma}))=(\mathcal{O}_\lambda,\sum_\epsilon c^\epsilon \mathcal{O}^\epsilon\cdot(1-\mathcal{O}^{s_\gamma}))=(\mathcal{O}_\lambda, \sum_\epsilon c^\epsilon \sum_\nu(-1)^{l(\nu/\epsilon)}\sqrt{J_\epsilon J_\nu}\mathcal{I}^\nu).$$
Here the second sum is over all $\nu\in W^P$ such that $\nu/\epsilon$ is a short rook strip.  
Using the linearity of the pairing:
$$\sum_\epsilon c^\epsilon(\mathcal{O}_\lambda,\sum_\nu(-1)^{l(\nu/\epsilon)}\sqrt{J_\epsilon J_\nu}\mathcal{I}^\nu)=\sum_\epsilon c^\epsilon\sum_\nu(-1)^{l(\nu/\epsilon)}\sqrt{J_\epsilon J_\nu}(\mathcal{O}_\lambda,\mathcal{I}^\nu).$$
Since $(\mathcal{O}_\lambda,\mathcal{I}^\nu)=\delta_{\lambda\nu}$,
we are left with 
\begin{equation}\label{E:sum-epsilon}\sum_\epsilon c^\epsilon\sum_\nu(-1)^{l(\nu/\epsilon)}\sqrt{J_\epsilon J_\nu}\delta_{\lambda\nu}=\sum_\epsilon c^\epsilon (-1)^{l(\lambda/\epsilon)}\sqrt{J_\epsilon J_\lambda} \end{equation}
where $\epsilon$ is such that $\epsilon/\mu$ is a short skew shape and $\lambda/\epsilon$ is a short rook strip. In particular,
$\mu \subset \epsilon \subset \lambda$. 

If $\mu = \lambda$, then $\epsilon = \mu$, $c^\mu =1/J_\mu$, and
\[ \sum_\epsilon c^\epsilon (-1)^{l(\lambda/\epsilon)}\sqrt{J_\epsilon J_\lambda}= c^\mu \cdot J_\mu \cdot (-1)^0 = 1 \/. \]

Now let $\mu \subsetneq \lambda$. Note that this will only occur in the case that $X$ is cominuscule and not minuscule. Let $A$ be the set of the partitions $\epsilon$ as above. 
Call a box, $\alpha$, \textit{removable} if $\alpha\in(\lambda/\epsilon)$ for some $\epsilon\in A$.
Fix a removable box and call it $\beta$. 
Define $\phi: A\rightarrow A$ such that 
$$\phi(\epsilon)= \left\{\begin{array}{lr}
        \epsilon\setminus \beta, & \text{if } \beta\in \epsilon\\
        \epsilon\cup \beta , &  \text{if } \beta\notin \epsilon
    \end{array}
   \right .$$
Once we have shown that $\phi$ is well-defined then it is an involution of $A$. 
$\phi$ is well-defined due to the condition that $\lambda/\epsilon$ is a short rook strip; there is a finite number of removable boxes and they are all mutually incomparable. 
$\beta$ is one such box, and thus both $\epsilon$ and $\phi(\epsilon)$ are in $A$. 
Furthermore, $l(\phi(\epsilon))=l(\epsilon)\pm 1$ and thus $l(\lambda/\phi(\epsilon))=l(\lambda/\epsilon)\mp 1$. 
Let $B\subset A$ be the set of $\epsilon$ that include $\beta$. 
Then we can use \Lemma{weight} above and the commutativity of the equivariant K-theory of a point to write 
$$\sum_{\epsilon\in A} c^\epsilon (-1)^{l(\lambda/\epsilon)}\sqrt{J_\epsilon J_\lambda}=\sum_{\epsilon\in A}  \frac{\sqrt{J_\mu J_\epsilon}}{J_\mu J_\epsilon}\sqrt{J_\epsilon J_\lambda}(-1)^{l(\lambda/\epsilon)}=\frac{\sqrt{J_\mu J_\lambda}}{J_\mu}\sum_{\epsilon\in A}(-1)^l(\lambda/\epsilon)$$
and 
$$\sum_{\epsilon\in A}(-1)^{l(\lambda/\epsilon)}=\sum_{\epsilon \in B}(-1)^{l(\lambda/\epsilon)}+(-1)^{l(\lambda/\phi(\epsilon))}=0.$$

Thus we have shown that 
$$(\mathcal{O}_\lambda,\alpha^\mu\cdot(1-\mathcal{O}^{s_\gamma}))=\sum_\epsilon c^\epsilon (-1)^{l(\lambda/\epsilon)}\sqrt{J_\epsilon J_\lambda}=\delta_{\lambda\mu}.$$

\end{proof}

In the case where $X$ is the Type A Grassmannian, \Theorem{classical} is consistent with \cite[Ex. 3.4.3 (2)]{brion:lectures}. Given $L$ is the ample generator of the Picard group of $X$, then the element dual to the Schubert class $\mathcal{O}^{\lambda}$, is the product $\mathcal{O}^{\lambda}\cdot[L^{-1}]$.

The basis for $K^T(X)$ of ideal sheaves $\{\mathcal{I}^\mu \ | \ \mu \in W^P\}$ will be dual to the Schubert basis $\{\mathcal{O}_\lambda \ | \ \lambda \in W^P\}$ with respect to the sheaf Euler characteristic pairing. 
However, the opposite Schubert basis, $\{\mathcal{O}^\lambda \ | \ \lambda \in W^P\}$, will also have a dual basis with respect to this pairing. 
This second dual basis, call it $\{\mathcal{I}_\mu \ | \ \mu \in W^P\}$, will also correspond to ideal sheaves. One can use the left Weyl group action to act by $w_0$, the longest element of $W$, to obtain a formula for the corresponding $\mathcal{I}_\mu$. 
This action is well defined at the level of Schubert cells, and thus respects the closures and boundaries that come up in the short exact sequence used to define the ideal sheaf. 
Thus 
$$w_0.\mathcal{I}^\mu=\mathcal{I}_{\mu^{\vee}}$$
if one is working non-equivariantly. 
However, this action is a ring homomoprhism, not an algebra homomorphism, and so in the equivariant setting one must be careful that there is a twisting of the weights \cite[Thm. 5.3, 5.5]{mihalcea.naruse.su:leftdemazure}. See also \cite{knutson:action} for more details. 
The same technique will work in the quantum $K$-theory setting as well. 

The formula for $\mathcal{I}^\mu$ in \Theorem{classical} is used in proving a very similar formula for the quantized ideal sheaf $\mathcal{I}^{\mu}_q\in QK_T(X)$.
However, \cite[Lemma 3.4]{buch.chaput.ea:qkchev} also gives a formula for $\mathcal{I}^\mu$. 
For a given $\mu\in W^p$, 
\begin{equation}\label{chevII}
    \mathcal{I}^\mu=\sum_{\nu/\mu}(-1)^{l(\nu/\mu)}\mathcal{O}^\nu
\end{equation} 
where $\nu/\mu$ is a rook strip.
Thus together with \Theorem{classical}, we have the following corollary.

\begin{cor}
    Let $\mu\in W^P$, then
    $$\alpha^{\mu}\cdot(1-\mathcal{O}^{s_\gamma})=\sum_{\nu/\mu}(-1)^{l(\nu/\mu)}\mathcal{O}^\nu$$
    where $\nu/\mu$ is a rook strip.
\end{cor}
This corollary tells us that $\alpha^\mu$ from \Theorem{classical} is the result of dividing the summation on the right hand side by the invertible element $(1-\mathcal{O}^{s_\gamma})$. 
The invertibility follows because this element is also equal to the product of line bundles 
$[\mathcal{L}^{\vee}_{\varpi_k} \otimes \mathbb{C}_{-\varpi_k}]$, where 
$\mathcal{L}_{\varpi_k}= G \times^P \mathbb{C}_{-\varpi_k}$ is the homogeneous line bundle associated to with the fundamental weight $\varpi_k$.
See also \eqref{E:Lvarpi} below.

\begin{example}
In this example we let the cominuscule space be $X=Gr(3,6)$.
We find $\alpha^\mu$ for a given $\mu\in W^P$ and demonstrate how \Theorem{classical} is consistent with \eqref{chevII}. 
Let $\mu\in W^P$ correspond to the partition $(3,1)$. 
Recall that $\mathcal{I}^\mu=\sum_{\nu/\mu}(-1)^{l(\nu/\mu)}\mathcal{O}^\nu$ where $\nu/\mu$ is a rook strip. 
We will find the various $\nu$ for this sum by looking at the diagram of $\mu=(3,1)$.
\begin{center}
\ytableausetup{smalltableaux}
\ydiagram{3,1}
\end{center}
The diagrams for $\nu$ correspond to $(3,1)$, $(3,1,1)$, $(3,2)$, and $(3,2,1)$,
\begin{center}
\hspace{2mm}\ydiagram{3,1} \hspace{2mm}
\ydiagram{3,1,1} \hspace{2mm} 
\ydiagram{3,2} \hspace{2mm} 
\ydiagram{3,2,1}
\end{center}
and thus $$\mathcal{I}^{(3,1)}=\mathcal{O}^{(3,1)}-\mathcal{O}^{(3,1,1)}-\mathcal{O}^{(3,2)}+\mathcal{O}^{(3,2,1)}$$
Since $X=Gr(3,6)$ is minuscule, it has no short roots, and so $\alpha^{(3,1)}=c^{(3,1)}\mathcal{O}^{(3,1)}$
If we calculate $\mathcal{O}^{(3,1)}\cdot(1-\mathcal{O}^{(1)})$ nonequivariantly, we find that $$\mathcal{O}^{(3,1)}\cdot\mathcal{O}^{(1)}=\mathcal{O}^{(3,1,1)}+\mathcal{O}^{(3,2)}-\mathcal{O}^{(3,2,1)}$$
and thus 
$$\mathcal{O}^{(3,1)}\cdot(1-\mathcal{O}^{(1)})=\mathcal{O}^{(3,1)}-\mathcal{O}^{(3,1,1)}-\mathcal{O}^{(3,2)}+\mathcal{O}^{(3,2,1)}=\mathcal{I}^{(3,1)}.$$
Thus \eqref{chevII} agrees with \Theorem{classical}. 
Note that we write this last calculation nonequivariantly because the equivariant calculation has dozens of terms and it becomes difficult to follow all the cancellations. 
\end{example}

\begin{example}
We go through a similar process as in the last example, but now in a space that is not minuscule.
Let the cominuscule space be $X=LG(4,8)$ and let $\mu\in W^P$ correspond to the strict partition $(3,2)$. 
Below we can see the shape of $\mu=(3,2)$.
\begin{center}
\ytableausetup{smalltableaux}
\ydiagram{3,1+2}
\end{center}
In order to find $\alpha^{\mu}$ we need all $\epsilon\in W^P$ such that $\epsilon/\mu$ is a short skew shape, or $\epsilon=\mu$. 
The possible $\epsilon$ where $\epsilon/\mu$ is a short skew shape, or $\epsilon = \mu$ have the following diagrams
\begin{center}
\ytableausetup{smalltableaux}
\ytableausetup{nobaseline}
\ydiagram{3,1+2} \hspace{5mm}
\ydiagram{4,1+2} \hspace{5mm}
\ydiagram{4,1+3}
\end{center}
Thus the possible $\epsilon$ correspond to the strict partitions $(3,2)$, $(4,2)$, and $(4,3)$.
Note that $(3,2,1)/(3,2)$ does not give a short skew shape, because the root corresponding to the remaining box is long. From \Theorem{classical}, 
$$\alpha^{(3,2)}=c^{(3,2)}\mathcal{O}^{(3,2)}+c^{(4,2)}\mathcal{O}^{(4,2)}+c^{(4,3)}\mathcal{O}^{(4,3)}$$
and
$$\mathcal{I}^{(3,2)}=(c^{(3,2)}\mathcal{O}^{(3,2)}+c^{(4,2)}\mathcal{O}^{(4,2)}+c^{(4,3)}\mathcal{O}^{(4,3)})\cdot(1-\mathcal{O}^{(1)}).$$
We restrict to the nonequivariant setting by setting all the $c^\epsilon$ to $1$. 
Then 
$$\mathcal{O}^{(3,2)}\cdot(1-\mathcal{O}^{(1)})=\mathcal{O}^{(3,2)}-2\mathcal{O}^{(4,2)}+\mathcal{O}^{(4,3)}-\mathcal{O}^{(3,2,1)}+2\mathcal{O}^{(4,2,1)}-\mathcal{O}^{(4,3,1)}$$
$$\mathcal{O}^{(4,2)}\cdot(1-\mathcal{O}^{(1)})=\mathcal{O}^{(4,2)}-2\mathcal{O}^{(4,3)}-\mathcal{O}^{(4,2,1)}+2\mathcal{O}^{(4,3,1)}$$
$$\mathcal{O}^{(4,3)}\cdot(1-\mathcal{O}^{(1)})=\mathcal{O}^{(4,3)}-\mathcal{O}^{(4,3,1)}.$$
Summing these together, we get $\mathcal{O}^{(3,2)}-\mathcal{O}^{(4,2)}-\mathcal{O}^{(3,2,1)}+\mathcal{O}^{(4,2,1)}=\mathcal{I}^{(3,2)}$. 

Alternatively, \eqref{chevII} gives that for a given $\mu\in W^p$, $\mathcal{I}^\mu=\sum_{\nu/\mu}(-1)^{l(\nu/\mu)}\mathcal{O}^\nu$ where $\nu/\mu$ is a rook strip. 
The possible $\nu$ have the following diagrams
\begin{center}
\ytableausetup{smalltableaux}
\ytableausetup{nobaseline}
\ydiagram{3,1+2} \hspace{5mm}
\ydiagram{4,1+2} \hspace{5mm}
\ydiagram{3,1+2,2+1} \hspace{5mm}
\ydiagram{4,1+2,2+1}
\end{center}
Thus the possible $\nu$ correspond to the strict partitions $(3,2)$, $(4,2)$, $(3,2,1)$, and $(4,2,1)$. 
Thus $\mathcal{I}^{(3,2)}=\mathcal{O}^{(3,2)}-\mathcal{O}^{(4,2)}-\mathcal{O}^{(3,2,1)}+\mathcal{O}^{(4,2,1)}$, and we see that the nonequivariant version of \Theorem{classical} agrees with \eqref{chevII}.
\end{example}

\begin{example}
Again let $X=LG(4,8)$ and $\mu=(3,2)$. 
We apply \eqref{chevI} three times, in order to calculate
$$\mathcal{I}^{(3,2)}=(c^{(3,2)}\mathcal{O}^{(3,2)}+c^{(4,2)}\mathcal{O}^{(4,2)}+c^{(4,3)}\mathcal{O}^{(4,3)})\cdot(1-\mathcal{O}^{(1)}).$$
For $\epsilon=(3,2)$, the only short rook strips $\nu/\epsilon$ are when $\nu_1=\epsilon=(3,2)$ and $\nu_2=(4,2)$. 
\begin{center}
\ydiagram{3,1+2} \hspace{2mm}-\hspace{2mm}
\ydiagram{4,1+2}
\end{center}
So
$$\mathcal{O}^{(3,2)}\cdot(1-\mathcal{O}^{(1)})= c^{(3,2)}\sqrt{J_{(3,2)}J_{(3,2)}}\mathcal{I}^{(3,2)}-c^{(3,2)}\sqrt{J_{(3,2)}J_{(4,2)}}\mathcal{I}^{(4,2)}$$
Similarly for $\epsilon'=(4,2)$, the only short rook strips are $\nu_3=\epsilon'=(4,2)$ and $\nu_4=(4,3)$. 
\begin{center}
    \hspace{2mm}\ydiagram{4,1+2} \hspace{2mm}-\hspace{2mm}
\ydiagram{4,1+3}
\end{center}
So 
$$\mathcal{O}^{(4,2)}\cdot(1-\mathcal{O}^{(1)})=c^{(4,2)}\sqrt{J_{(4,2)}J_{(4,2)}}\mathcal{I}^{(4,2)}-c^{(4,2)}\sqrt{J_{(4,2)}  J_{(4,3)}}\mathcal{I}^{(4,3)}$$
Lastly, when $\epsilon''=(4,3)$ only $\nu_5=\epsilon''=(4,3)$ will give a short rook strip.
\begin{center}
\hspace{2mm}\ydiagram{4,1+3}
\end{center}
So 
$$\mathcal{O}^{(4,3)}\cdot(1-\mathcal{O}^{(1)})=c^{(4,3)}\sqrt{J_{(4,3)}J_{(4,3)}}\mathcal{I}^{(4,3)}$$

Thus 
\[ \begin{split}
(c^{(3,2)}&\mathcal{O}^{(3,2)}+c^{(4,2)}\mathcal{O}^{(4,2)}+c^{(4,3)}\mathcal{O}^{(4,3)})\cdot(1-\mathcal{O}^{(1)}) \\ 
 & = c^{(3,2)}\sqrt{J_{(3,2)}J_{(3,2)}}\mathcal{I}^{(3,2)}-c^{(3,2)}\sqrt{J_{(3,2)}J_{(4,2)}}\mathcal{I}^{(4,2)} +c^{(4,2)}\sqrt{J_{(4,2)}J_{(4,2)}}\mathcal{I}^{(4,2)} \\
 & \hspace{3mm}-c^{(4,2)}\sqrt{J_{(4,2)}  J_{(4,3)}}\mathcal{I}^{(4,3)} +c^{(4,3)}\sqrt{J_{(4,3)}J_{(4,3)}}\mathcal{I}^{(4,3)} \\
 & = \mathcal{I}^{(3,2)}-\frac{\sqrt{J_{(3,2)}J_{(4,2)}}}{J_{(3,2)}}\mathcal{I}^{(4,2)} +\frac{\sqrt{J_{(3,2)}J_{(4,2)}}}{J_{(3,2)}}\mathcal{I}^{(4,2)} \\
& \hspace{3mm} -\frac{\sqrt{J_{(3,2)} J_{(4,3)}}}{J_{(3,2)}}\mathcal{I}^{(4,3)} +\frac{\sqrt{J_{(3,2)}J_{(4,3)}}}{J_{(3,2)}}\mathcal{I}^{(4,3)} \\
& = \mathcal{I}^{(3,2)}+\frac{\sqrt{J_{(3,2)}J_{(4,2)}}}{J_{(3,2)}}(-\mathcal{I}^{(4,2)}+\mathcal{I}^{(4,2)})+\frac{\sqrt{J_{(3,2)}J_{(4,3)}}}{J_{(3,2)}}(-\mathcal{I}^{(4,3)}+\mathcal{I}^{(4,3)}) \\
& = \mathcal{I}^{(3,2)}
\end{split}
\]
and we can see that \Theorem{classical} holds.

Using \Lemma{weight} to mitigate the equivariant coefficients, we can see how the diagrams of the various $\nu$ from \eqref{chevI} will index the ideal sheaves that cancel.
\begin{center}
\big(\hspace{2mm}\ydiagram{3,1+2} \hspace{2mm}-\hspace{2mm}
\ydiagram{4,1+2} \hspace{2mm}\big) + \big(\hspace{2mm}\ydiagram{4,1+2} \hspace{2mm}-\hspace{2mm}
\ydiagram{4,1+3} \hspace{2mm}\big) + \hspace{2mm}\ydiagram{4,1+3}\hspace{2mm} =\hspace{2mm} \ydiagram{3,1+2}
\end{center}
\end{example}

\section{Quantum Dual Basis}\label{sec:quantum results}
We will use \Theorem{classical} to prove a similar formula for the quantized ideal sheaves.
Recall the quantized ideal sheaves are defined to form a basis dual to the Schubert classes under the quantum $K$-metric on $QK_T(X)$. 
Before we prove this formula we will need a lemma that discusses the Gromov Witten invariants that will appear. 

\begin{lemma}\label{lemma:GWi}
Let $\lambda\in W^P$ and let $\kappa\in K^T(X)$. 
Then for all $d\geq 1$ we have 
$$I_d(\mathcal{O}_\lambda,\kappa,(1-\mathcal{O}^{s_\gamma}))=0.$$
\end{lemma}
\begin{proof}
    First note that 
    $$I_d( \mathcal{O}_\lambda,\kappa, (1-\mathcal{O}^{s_\gamma}))=I_d(\mathcal{O}_\lambda,(1-\mathcal{O}^{s_\gamma}),\kappa)=I_d( \mathcal{O}_\lambda,1,\kappa)-I_d( \mathcal{O}_\lambda, \mathcal{O}^{s_\gamma},\kappa).$$
Furthermore, $\kappa\in K^T(X)$ and thus can be expressed in terms of the ideal sheaf basis, $$\kappa=\sum_w \beta^w\mathcal{I}^w$$ for $w\in W^P$ and $\beta^w\in K^T(\text{point})$.
Thus we can rewrite the Gromov Witten invariants using linearity
$$I_d(\mathcal{O}_\lambda,(1-\mathcal{O}^{s_\gamma}),\sum_w \beta^w\mathcal{I}^w)=\sum_w \beta^w \big( I_d(\mathcal{O}_\lambda,1,\mathcal{I}^w)-I_d(\mathcal{O}_\lambda,\mathcal{O}^{s_\gamma},\mathcal{I}^w)\big).$$
Note that all the curves in $X$ meet the Schubert divisor $X^{s_\gamma}$ and thus $\Gamma_d(X_\lambda,X^{s_\gamma})=\Gamma_d(X_\lambda)=\Gamma_d(X_\lambda,X)$ for all $d\geq1$. \cite[Cor. 4.2]{buch.chaput.ea:projected} then gives for $u,v\in W^P$
$$[\mathcal{O}_{\Gamma_d(X_u,X^v)}]=\sum_{w\in W^P} I_d(\mathcal{O}_u,\mathcal{O}^v,\mathcal{I}^w)\mathcal{O}^w.$$
Thus since $\Gamma_d(X_\lambda,X^{s_\gamma})=\Gamma_d(X_\lambda,X)$, the Gromov Witten invariants are also equal and we have
$$I_d(\mathcal{O}_\lambda,1,\mathcal{I}^w)-I_d(\mathcal{O}_\lambda,\mathcal{O}^{s_\gamma},\mathcal{I}^w)=0$$
for all $\lambda,w\in W^P$ and $d\geq 1$. 
Thus $$I_d( \mathcal{O}_\lambda,\kappa, (1-\mathcal{O}^{s_\gamma}))=\sum_w \beta^w \big( 0 \big)=0$$
for $d\geq 1.$
\end{proof}

Note that using linearity we can also replace $\mathcal{O}_\lambda$ with any element of $K^T(X)$, and the Gromov Witten invariants will vanish in the same way. 
The lemma is only written in this manner for simplicity within its proof and so that it can be applied in a direct manner in the next theorem. 
Also note that this lemma is a special case of the ``quantum equals classical" results from \cite{chaput.perrin:rationality}, as well as those from \cite{buch.mihalcea:quantum}, specifically Remark 6.5.

We can now prove the first of two formulae for the quantized ideal sheaves. 

\begin{thm}\label{thm:quantum}
Let $\mathcal{I}_q^\mu=\alpha^{\mu}\star(1-\mathcal{O}^{s_\gamma})$ where $\alpha^\mu$ is defined the same as in \Theorem{classical}. Then the collection $\{\mathcal{I}_q^\mu \ | \ \mu \in W^P\}$ forms the dual basis with respect to the Schubert basis, $\{\mathcal{O}_\lambda \ | \ \lambda \in W^P\}$, under the quantum $K$-metric for $QK_T(X)$ when $X$ is cominuscule, i.e. 
$$\left(\!\left(\mathcal{O}_\lambda,\mathcal{I}_q^\mu\right)\!\right)=\delta_{\lambda\mu}.$$
\end{thm}

\begin{proof}
By definition of the quantum $K$-metric 
$$\left(\!\left(\mathcal{O}_\lambda ,\alpha^\mu \star (1-\mathcal{O}^{s_\gamma})\right)\!\right)=\sum_{d\geq0}q^dI_d(\mathcal{O}_\lambda,\alpha^\mu, (1-\mathcal{O}^{s_\gamma}))=I_0( \mathcal{O}_\lambda,\alpha^\mu, (1-\mathcal{O}^{s_\gamma}))$$
where \Lemma{GWi} gives that the Gromov Witten invariants of degree $1$ or more vanish. 
Here the remaining degree $0$ Gromov Witten invariant is the sheaf Euler characteristic, and from \Theorem{classical}, $\mathcal{I}^\mu=\alpha^\mu \cdot (1-\mathcal{O}^{s_\gamma})$, so we have
$$I_0( \mathcal{O}_\lambda,\alpha^\mu, (1-\mathcal{O}^{s_\gamma}))=\int_X\mathcal{O}_\lambda\cdot\alpha^\mu\cdot (1-\mathcal{O}^{s_\gamma})=\int_X\mathcal{O}_\lambda\cdot\mathcal{I}^\mu=\delta_{\lambda\mu}.$$ 
Therefore $$\left(\!\left(\mathcal{O}_\lambda,\alpha^\mu\star (1-\mathcal{O}^{s_\gamma})\right)\!\right)=\delta_{\lambda\mu}.$$    
\end{proof}

Again, there will also be an opposite quantized ideal sheaf basis $\{\mathcal{I}_\mu^q \ | \ \mu \in W^P\}$, which is dual to the opposite Schubert basis, $\{\mathcal{O}^\lambda \ | \ \lambda \in W^P\}$, under the quantum $K$-metric on $QK_T(X)$. 
We can use the same technique as with the classical ideal sheaves to obtain a formula for $\mathcal{I}_\mu^q$ from $\mathcal{I}^\mu_q$ by making use of the left Weyl group action by the element $w_0$. 
See \cite[Section 8]{mihalcea.naruse.su:leftdemazure} for more information.

The Chevalley formula for $QK_T(X)$ given in \cite[Thm. 3.9]{buch.chaput.ea:qkchev} for $X$ cominuscule, allows us to write another formula for $\mathcal{I}_q^\mu$. 
From the Chevalley formula 
$$\mathcal{O}^{u}\star (1-\mathcal{O}^{s_\gamma})=\mathcal{O}^u\cdot (1-\mathcal{O}^{s_\gamma})-q\psi(\mathcal{O}^u\cdot (1-\mathcal{O}^{s_\gamma}))$$ 
where $\psi:K^T(X)\to K^T(X)$ is the homomorphism of $\Gamma$-modules defined by taking the first curve neighborhood, $\psi(\mathcal{O}^u)=\mathcal{O}^{u(-1)}$. 
Thus we can take sums of such products and use linearity to write products involving $\alpha^\mu$ instead of $\mathcal{O}^u$. 
The left hand side becomes $\mathcal{I}_q^\mu$ and the right hand side becomes $\mathcal{I}^\mu$ minus $q$ times the sum of the first curve neighborhoods of the Schubert classes that compose $\mathcal{I}^\mu$. 
This term can be thought of as a quantum correction that vanishes when restricting $q$ to $0$, to recover the classical ideal sheaf from the quantized ideal sheaf. 
Let $\psi \mathcal{I}^u$ denote this quantum correction. 

For a given $\mu\in W^p$, \eqref{chevII} gives that $\mathcal{I}^\mu=\sum_{\nu/\mu}(-1)^{l(\nu/\mu)}\mathcal{O}^\nu$ where $\nu/\mu$ is a rook strip \cite[Lemma 3.4]{buch.chaput.ea:qkchev}, thus \begin{equation}\label{qcorrection}\psi\mathcal{I}^\mu=\psi\sum_{\nu/\mu}(-1)^{l(\nu/\mu)}\mathcal{O}^\nu=\sum_{\nu/\mu}(-1)^{l(\nu/\mu)}\psi\mathcal{O}^{\nu}=\sum_{\nu/\mu}(-1)^{l(\nu/\mu)}\mathcal{O}^{\nu(-1)}\end{equation} by linearity of $\psi$.

In practice, the $\psi\mathcal{I}^\mu$ will often be 0. 
When $\psi\mathcal{I}^\mu$ is nonzero, \eqref{qcorrection} tells us that it is an alternating sum of Schubert classes that is cancellation free. 
This is because \eqref{chevII} gives that, for a fixed $\mu\in W^P$, the ideal sheaf $\mathcal{I}^\mu$ will never be $0$ as we count $\nu/\mu$ to a rook strip when $\nu=\mu$. 
Any other $\nu\in W^P$ where $\nu/\mu$ is a rook strip will then be distinct, so despite the alternating sign, the Schubert classes in the sum for $\mathcal{I}^\mu$ in \eqref{chevII} never cancel with each other. 
However, distinct $\nu\in W^P$ can have the same first curve neighborhood, $\nu(-1)$, thus it is possible to get the same Schubert class with both a positive and negative coefficient in the above formula for $\psi\mathcal{I}^\mu$. If this occurs the next theorem shows that the entire sum for $\psi\mathcal{I}^\mu$ becomes $0$.
Recall $z_1$ is used to find $\nu(-1)$. This element $z_1$ is defined in \Section{curve} and can be seen highlighted in \Table{tablez1}.

\begin{thm}\label{thm:Thm3}
For $\mu\in W^P$,  $\mathcal{I}^\mu_q=\mathcal{I}^{\mu}-q\psi\mathcal{I}^{\mu}.$ Furthermore if $z_1\nleqslant \mu$, then $\psi\mathcal{I}^\mu=0.$ If $z_1\leq \mu$ then $\psi\mathcal{I}^{\mu}=\mathcal{I}^{\mu(-1)}.$
\end{thm}

\begin{proof} As described above, we begin with the Chevalley formula in $QK_T(X)$
$$\mathcal{O}^{u}\star (1-\mathcal{O}^{s_\gamma})=\mathcal{O}^u\cdot (1-\mathcal{O}^{s_\gamma})-q\psi(\mathcal{O}^u\cdot (1-\mathcal{O}^{s_\gamma}))$$ 
and proceed by taking sums on both sides of the equation
$$\sum_\epsilon c^\epsilon \mathcal{O}^\epsilon \star (1-\mathcal{O}^{s_\gamma})= \sum_\epsilon c^\epsilon \mathcal{O}^\epsilon \cdot (1-\mathcal{O}^{s_\gamma}) - q \psi (\sum_\epsilon c^\epsilon \mathcal{O}^\epsilon \cdot (1-\mathcal{O}^{s_\gamma}))$$
where $\epsilon/\mu$ is a short skew shape, or $\epsilon=\mu$. Thus using \Theorem{quantum} on the right hand side, as well as \Theorem{classical} and the definition of $\psi\mathcal{I}^\mu$ on the left hand side, we conclude that
$$\mathcal{I}^\mu_q=\mathcal{I}^{\mu}-q\psi\mathcal{I}^{\mu}.$$

Next assume that $z_1\nleqslant\mu$. 
Fix $\beta \in z_1 \setminus \mu$ to be a box such that $\mu\cup\beta\in W^P$. 
Let $A^\mu=\{\lambda\in W^P: \lambda/\mu \text{ is a rook strip}\}$. 
Define $\phi: A^\mu \to A^\mu$ such that 
$$\phi(\lambda)= \left\{\begin{array}{lr}
        \lambda\setminus \beta, & \text{if } \beta\in \lambda\\
        \lambda\cup \beta , &  \text{if } \beta\notin \lambda
    \end{array}
   \right .$$
Once we have shown that $\phi$ is well-defined then it is an involution on $A^\mu$. 
$\phi$ is well-defined of because of the rook strip condition on $\lambda$; there are a finite number of boxes that can be added to $\mu$ that result in a $\lambda\in A^\mu$, and all such boxes are mutually incomparable.
Since $\beta$ is one such box, both $\lambda$ and $\phi(\lambda)$ lie in $A^{\mu}$. 
Since $\beta\in z_1$, $\psi(\mathcal{O}^{\phi(\lambda)})=\psi(\mathcal{O}^{\lambda})=\mathcal{O}^{\lambda(-1)}$.
 Further $l(\phi(\lambda))=l(\lambda)\pm 1$, so 
 $$(-1)^{l(\lambda/\mu)}\psi\mathcal{O}^{(\lambda)}+(-1)^{l(\phi(\lambda)/\mu)}\psi\mathcal{O}^{(\phi(\lambda))}=\mathcal{O}^{\lambda(-1)}-\mathcal{O}^{\lambda(-1)}=0.$$
 Since this holds for all $\lambda \in A^\mu$ we have that 
$$\psi\mathcal{I}^\mu=\sum_{\lambda\in A^\mu}(-1)^{l(\lambda/\mu)}\mathcal{O}^{\lambda(-1)}=0.$$

Assume $z_1\leq \mu$. Then from \eqref{qcorrection},
$$\psi\mathcal{I}^\mu=\sum_{\nu/\mu}(-1)^{l(\nu/\mu)}\mathcal{O}^{\nu(-1)}.$$ 
Since $z_1\leq \mu$, all the $\nu(-1)$ are distinct in this summation. 
From \eqref{chevII}, we have that 
$$\mathcal{I}^{\mu(-1)}=\sum_{\epsilon/\mu(-1)}(-1)^{l(\epsilon/\mu(-1))}\mathcal{O}^{\epsilon}.$$
Again, since $z_1\leq \mu$, there is then a bijection between the boxes that can be added to $\mu$ that result in a rook strip, and the boxes that can be added to $\mu(-1)$ that result in a rook strip. 
Thus the set of $\nu(-1)$ and the set of $\epsilon$ are identical. 
This bijection of boxes also ensures that the associated lengths $l(\nu/\mu)$ and $l(\epsilon/\mu(-1))$ are equal. 
Thus we conclude that the sums are identical, and we have that $\psi\mathcal{I}^\mu=\mathcal{I}^{\mu(-1)}$.

\end{proof}

 Recall in the minuscule cases, the sums 
$$\sum_\epsilon c^\epsilon \mathcal{O}^\epsilon \star (1-\mathcal{O}^{s_\gamma})= \sum_\epsilon c^\epsilon \mathcal{O}^\epsilon \cdot (1-\mathcal{O}^{s_\gamma}) - q \psi (\sum_\epsilon c^\epsilon \mathcal{O}^\epsilon \cdot (1-\mathcal{O}^{s_\gamma}))$$
reduce to the case where $\epsilon=\mu$. 
\Theorem{Thm3} is significant because both $\mathcal{I}^\mu$ and $\psi\mathcal{I}^\mu$ have expressions in terms of Schubert classes that can be found without calculating any products in either $K^T(X)$ or $QK_T(X)$. For a given $\mu\in W^p$, \eqref{chevII} and $\eqref{qcorrection}$ provide a way to write $\mathcal{I}^\mu_q$ in terms of Schubert classes. 
Since it is a simple combinatorial procedure to find $\mathcal{O}^{\nu(-1)}$ for a given $\nu\in W^P$, \Theorem{Thm3} gives a strictly combinatorial formula for $\mathcal{I}^\mu_q$ in terms that can be calculated using the diagram for $\mu$. 

We will demonstrate finding $\mathcal{I}^\mu_q$ in the next examples. 
Since we wish to work just with diagrams, for the remainder of the section the Weyl group elements $\mu,\nu\in W^P$ will be identified with their corresponding diagrams. 
Recall that in type A, where $X= G/P =Gr(k,n)$, $\nu(-1)$ will correspond to the partition $\nu$ with the top row and left most column removed, with each box moved one step north-west. 
In order for $\nu/\mu$ to be a rook strip, $\nu$ can add at most $1$ box to $\mu$ in each row and column.

\begin{example} Let $X=Gr(3,6)$, and let $\mu=(2,2,2)$. 
\begin{center}
    \ydiagram[*(white) \bullet]
    {2+0,2+0,2+0}
    *[*(white)]{2,2,2}
\end{center}
The only $\nu$ possible such that $\nu/\mu$ is a rook strip are $\nu_1=(2,2,2)$ and $\nu_2=(3,2,2)$. 
Here we have highlighted the box $\beta$ from \Theorem{Thm3} with a dot.
\begin{center}
    \ydiagram[*(white) \bullet]
    {2+0,2+0,2+0}
    *[*(white)]{2,2,2}
    \hspace{5mm}
    \ydiagram[*(white) \bullet]
    {2+1,2+0,2+0}
    *[*(white)]{3,2,2}
\end{center}
Thus $$\mathcal{I}^{(2,2,2)}=\mathcal{O}^{(2,2,2)}-\mathcal{O}^{(3,2,2)}.$$ 
We can see that $\nu_1(-1)=(1,1)$ and $\nu_2(-1)=(1,1)$, 
\begin{center}
\ydiagram{2,2,2}
    *[*(gray)]{2,1,1}
    $\xmapsto{\psi}$ 
    \ydiagram[*(white) \bullet]
    {1+0,1+0}
    *[*(white)]{1,1}
\hspace{5mm}
    \ydiagram[*(gray) \bullet]
    {2+1,2+0,2+0}
    *[*(white)]{0+0,1+1,1+1}
    *[*(gray)]{2,1,1}
    $\xmapsto{\psi}$ 
    \ydiagram[*(white) \bullet]
    {1+0,1+0}
    *[*(white)]{1,1}
\end{center}
and thus $$\psi\mathcal{I}^{(2,2,2)}=\mathcal{O}^{(1,1)}-\mathcal{O}^{(1,1)}=0.$$
Putting these two calculations together, we have that 
$$\mathcal{I}^{(2,2,2)}_q=\mathcal{O}^{(2,2,2)}-\mathcal{O}^{(3,2,2)}.$$
\end{example}

Still in Type A, $X=Gr(3,6),$ if it is not possible to add a box to the first row or the first column of $\mu$, then all the potential rook strips $\nu/\mu$ will give disctinct $\nu(-1)$ and the sum that composes $\psi\mathcal{I}^\mu$ will be identical to that of $\mathcal{I}^{\mu(-1)}$.

\begin{example}
Let $\mu=(3,2,1).$ 
\begin{center}
\ydiagram[*(white) \bullet]
    {3+0,2+0,1+0}
    *[*(white)]{3,2,1}
\end{center}
The possible $\nu$ such that $\nu/\mu$ is a rook strip are $\nu_1=(3,2,1)$, $\nu_2=(3,3,1)$, $\nu_3=(3,2,2)$, and $\nu_4=(3,3,2)$. Here we have highlighted all the boxes added to $\mu$ with a dot, but note that none of them are contained in $z_1$. 
Thus we know from \Theorem{Thm3} that $\psi\mathcal{I}^{(3,2,1)}\neq 0$.
\begin{center}
\ydiagram[*(white) \bullet]
    {3+0,2+0,1+0}
    *[*(white)]{3,2,1}
\hspace{5mm}
\ydiagram[*(white) \bullet]
    {3+0,2+1,1+0}
    *[*(white)]{3,3,1}
\hspace{5mm}
\ydiagram[*(white) \bullet]
    {3+0,2+0,1+1}
    *[*(white)]{3,2,2}
\hspace{5mm}
\ydiagram[*(white) \bullet]
    {3+0,2+1,1+1}
    *[*(white)]{3,3,2}   
\end{center}
Thus $\eqref{chevII}$ gives $$\mathcal{I}^{(3,2,1)}=\mathcal{O}^{(3,2,1)}-\mathcal{O}^{(3,3,1)}-\mathcal{O}^{(3,2,2)}+\mathcal{O}^{(3,3,2)}.$$ 
Taking the first curve neighborhood $\nu_1(-1)=(1)$, $\nu_2(-1)=(2)$, $\nu_3(-1)=(1,1)$ and $\nu_4(-1)=(2,1)$. 
The signs in $\psi\mathcal{I}^{(3,2,1)}$ are determined by the number of dots in the diagrams above.
\begin{center}
\ydiagram[*(white) \bullet]
    {3+0,2+0,1+0}
    *[*(gray)]{3,1,1}
    *[*(white)]{3,2,1}
   $\xmapsto{\psi}$
\ydiagram[*(white) \bullet]
    {1+0}
    *[*(white)]{1}
\hspace{2mm}
\ydiagram[*(white) \bullet]
    {3+0,2+1,1+0}
    *[*(gray)]{3,1,1}
    *[*(white)]{3,3,1}
   $\xmapsto{\psi}$
\ydiagram[*(white) \bullet]
    {1+1}
    *[*(white)]{2}
\hspace{2mm}
\ydiagram[*(white) \bullet]
    {3+0,2+0,1+1}
    *[*(gray)]{3,1,1}
    *[*(white)]{3,2,2}
   $\xmapsto{\psi}$
\ydiagram[*(white) \bullet]
    {1+0,0+1}
    *[*(white)]{1,1}
\hspace{2mm}
\ydiagram[*(white) \bullet]
    {3+0,2+1,1+1}
    *[*(gray)]{3,1,1}
    *[*(white)]{3,3,2}
   $\xmapsto{\psi}$
\ydiagram[*(white) \bullet]
    {1+1,0+1}
    *[*(white)]{2,1}
\end{center}
Thus we have $$\psi\mathcal{I}^{(3,2,1)}=\mathcal{O}^{(1)}-\mathcal{O}^{(2)}-\mathcal{O}^{(1,1)}+\mathcal{O}^{(2,1)}.$$
\Theorem{Thm3} then tells us that 
$$\mathcal{I}_q^{(3,2,1)}=\mathcal{O}^{(3,2,1)}-\mathcal{O}^{(3,3,1)}-\mathcal{O}^{(3,2,2)}+\mathcal{O}^{(3,3,2)}-q(\mathcal{O}^{(1)}-\mathcal{O}^{(2)}-\mathcal{O}^{(1,1)}+\mathcal{O}^{(2,1)}).$$
In this example, $z_1=(3,1,1)\leq (3,2,1)=\mu$ in $Gr(3,6)$, so we know $\psi\mathcal{I}^{\mu}$ should agree with $\mathcal{I}^{\mu(-1)}$. If we calculate $\mathcal{I}^{(1)}$ we then find that $(1)/(1)$, $(2)/(1)$, $(1,1)/(1)$, and $(2,1)/(1)$ give the only rook strips. We then have that 
$$\mathcal{I}^{(3,2,1)(-1)}=\mathcal{O}^{(1)}-\mathcal{O}^{(2)}-\mathcal{O}^{(1,1)}+\mathcal{O}^{(2,1)}=\psi\mathcal{I}^{(3,2,1)}.$$
Thus the partitions that index $\psi\mathcal{I}^{(3,2,1)}$ look identical to those of $\mathcal{I}^{(1)}$. 
\end{example}

In type C, where $X=G/P=LG(n,2n)$, something similar occurs.
Here $\nu(-1)\in W^P$  is the strict partition $\nu$ with the top row removed and boxes moved one step to the north-west.
Again, in order for $\nu/\mu$ to be a rook strip, $\nu$ can add at most $1$ box to $\mu$ in each row and column. 
If $\nu$ only adds a box to the top most row, $\nu(-1)$ and $\mu(-1)$ will be identical, even though $l(\nu/\mu)=1$ and $l(\mu/\mu)=0$, so the corresponding Schubert classes in $\psi\mathcal{I}^u$ will cancel. 
\begin{example} Let $X=LG(4,8)$, and let $\mu=(3,2)$. 
\begin{center}
    \ydiagram[*(white)\bullet]
    {3+0,2+0}
    *[*(white)]{3,1+2}
\end{center}
The possible $\nu$ such that $\nu/\mu$ is a rook strip are $\nu_1=(3,2)$, $\nu_2=(4,2)$, $\nu_3=(3,2,1)$, and $\nu_4=(4,2,1)$. 
Here we have highlighted all the boxes added to $\mu$ with a dot, but note that the highlighted box in the top row belongs to $z_1$. 
Thus we know from \Theorem{Thm3} that $\psi\mathcal{I}^{(3,2)}= 0$. 
We show the cancellation below.
\begin{center}
    \ydiagram[*(white)\bullet]
    {3+0,2+0}
    *[*(white)]{3,1+2}
    \hspace{5mm}
    \ydiagram[*(white)\bullet]
    {3+1,4+0}
    *[*(white)]{3,1+2}
    \hspace{5mm}
    \ydiagram[*(white)\bullet]
    {3+0,2+0,2+1}
    *[*(white)]{3,1+2,2+1}
    \hspace{5mm}
    \ydiagram[*(white)\bullet]
    {3+1,4+0,2+1}
    *[*(white)]{3,1+2,2+1}
    \hspace{5mm}
\end{center}
The above diagrams and \eqref{chevII} give that $$\mathcal{I}^{(3,2)}=\mathcal{O}^{(3,2)}-\mathcal{O}^{(4,2)}-\mathcal{O}^{(3,2,1)}+\mathcal{O}^{(4,2,1)}.$$ 
Taking the first curve neighborhood, we have $\nu_1(-1)=(2)$, $\nu_2(-1)=(2)$, $\nu_3(-1)=(2,1)$, and $\nu_4(-1)=(2,1)$. 
The sign in $\psi\mathcal{I}^{(3,2)}$ is then determined by the number of highlighted dots in the diagrams above. 
\begin{center}
    \ydiagram{0+0,1+2}
    *[*(gray)]{3,0}
\hspace{0.2mm}
    $\xmapsto{\psi}$
    \ydiagram{2}
\hspace{2mm}
    \ydiagram{0+0,1+2}
    *[*(gray)\bullet]{3+1,0}
    *[*(gray)]{3,0}
\hspace{0.2mm}
    $\xmapsto{\psi}$
    \ydiagram{2}
\hspace{2mm}
    \ydiagram[*(white)\bullet]
    {0+0,0+0,2+1}
    *[*(white)]{0+0,1+2,2+1}
    *[*(gray)]{3,0,0}
    \hspace{0.2mm}
    $\xmapsto{\psi}$
    \ydiagram[*(white)\bullet]
    {2+0,1+1}
    *[*(white)]{2,1+1}
\hspace{2mm}
    \ydiagram[*(white)\bullet]
    {3+0,3+0,2+1}
    *[*(white)]{0+0,1+2,2+1}
    *[*(gray)\bullet]{3+1,0,0}
    *[*(gray)]{3,0,0}
\hspace{0.2mm}
    $\xmapsto{\psi}$
    \ydiagram[*(white)\bullet]
    {2+0,1+1}
    *[*(white)]{2,1+1}
\end{center}
Thus $$\psi\mathcal{I}^{(3,2)}=\mathcal{O}^{(2)}-\mathcal{O}^{(2)}-\mathcal{O}^{(2,1)}+\mathcal{O}^{(2,1)}=0.$$
\Theorem{Thm3} then tells us that 
$$\mathcal{I}_q^{(3,2)}=\mathcal{O}^{(3,2)}-\mathcal{O}^{(4,2)}-\mathcal{O}^{(3,2,1)}+\mathcal{O}^{(4,2,1)}.$$
\end{example}

\begin{example}
Let $X=LG(4,8)$, and let $\mu=(4,2)$. 
\begin{center}
    \ydiagram[*(white)\bullet]
    {4+0,2+0,}
    *[*(white)]{4,1+2}
\end{center}
The possible $\nu$ such that $\nu/\mu$ is a rook strip are $\nu_1=(4,2)$, $\nu_2=(4,3)$, $\nu_3=(4,2,1)$, and $\nu_4=(4,3,1)$. 
Here we have highlighted all the boxes added to $\mu$ with a dot, but note that none of them are contained in $z_1$. 
Thus we know from \Theorem{Thm3} that $\psi\mathcal{I}^{(4,2)}\neq 0$.
\begin{center}
    \ydiagram[*(white)\bullet]
    {4+0,2+0}
    *[*(white)]{4,1+2}
    \hspace{5mm}
    \ydiagram[*(white)\bullet]
    {4+0,3+1}
    *[*(white)]{4,1+3}
    \hspace{5mm}
    \ydiagram[*(white)\bullet]
    {4+0,2+0,2+1}
    *[*(white)]{4,1+2,2+1}
    \hspace{5mm}
    \ydiagram[*(white)\bullet]
    {4+0,3+1,2+1}
    *[*(white)]{4,1+3,2+1}
    \hspace{5mm}
\end{center}
Thus $$\mathcal{I}^{(4,2)}=\mathcal{O}^{(4,2)}-\mathcal{O}^{(4,3)}-\mathcal{O}^{(4,2,1)}+\mathcal{O}^{(4,3,1)}.$$
Taking the first curve neighborhood, we have $\nu_1(-1)=(2)$, $\nu_2(-1)=(3)$, $\nu_3(-1)=(2,1)$, and $\nu_4(-1)=(3,1)$. 
Again, the sign in $\psi\mathcal{I}^{(4,2)}$ is determined by the number of dots in the diagrams above.
\begin{center}
    \ydiagram{0,1+2}
    *[*(gray)]{4}
    \hspace{0.5mm}
    $\xmapsto{\psi}$
    \ydiagram{2}
\hspace{2mm}
    \ydiagram[*(white)\bullet]
    {4+0,3+1}
    *[*(white)]{0,1+2}
    *[*(gray)]{4}
    $\xmapsto{\psi}$
    \ydiagram[*(white)\bullet]
    {2+1}
    *[*(white)]{2}
\hspace{2mm}
    \ydiagram[*(white)\bullet]
    {4+0,3+0,2+1}
    *[*(white)]{0,1+2,2+1}
    *[*(gray)]{4}
    $\xmapsto{\psi}$
    \ydiagram[*(white)\bullet]
    {2+0,1+1}
    *[*(white)]{2}
\hspace{1mm}
    \ydiagram[*(white)\bullet]
    {4+0,3+1,2+1}
    *[*(white)]{0,1+3,2+1}
    *[*(gray)]{4}
    $\xmapsto{\psi}$
    \ydiagram[*(white)\bullet]
    {2+1,1+1}
    *[*(white)]{3,1+1}
\end{center}
Thus $$\psi\mathcal{I}^{(4,2)}=\mathcal{O}^{(2)}-\mathcal{O}^{(3)}-\mathcal{O}^{(2,1)}+\mathcal{O}^{(3,1)}$$
and \Theorem{Thm3} gives that 
$$\mathcal{I}_q^{(4,2)}=\mathcal{O}^{(4,2)}-\mathcal{O}^{(4,3)}-\mathcal{O}^{(4,2,1)}+\mathcal{O}^{(4,3,1)}-q(\mathcal{O}^{(2)}-\mathcal{O}^{(3)}-\mathcal{O}^{(2,1)}+\mathcal{O}^{(3,1)}).$$
In this example, $z_1=(4)\leq (4,2)=\mu$ in $LG(4,8)$, so we know $\psi\mathcal{I}^{\mu}$ should agree with $\mathcal{I}^{\mu(-1)}$. If we calculate $\mathcal{I}^{(2)}$ we then find that $(2)/(2)$, $(3)/(2)$, $(2,1)/(2)$, and $(3,1)/(2)$ give the only rook strips. We then have that 
$$\mathcal{I}^{(4,2)(-1)}=\mathcal{O}^{(2)}-\mathcal{O}^{(3)}-\mathcal{O}^{(2,1)}+\mathcal{O}^{(3,1)}=\psi\mathcal{I}^{(4,2)}.$$
Thus the partitions that index $\psi\mathcal{I}^{(4,2)}$ look identical to those of $\mathcal{I}^{(2)}$.

\end{example}

Furthermore, this description that $\mathcal{I}^\mu_q=\mathcal{I}^{\mu}-q\psi\mathcal{I}^{\mu}$ from \Theorem{Thm3} also gives another interpretation of the Chevalley formula. 
Recall from \eqref{chevI} the classical product can be written as an alternating sum of classical ideal sheaves \cite[Thm. 3.6]{buch.chaput.ea:qkchev}: 
$$\mathcal{O}^\mu\cdot(1-\mathcal{O}^{s_\gamma})=\sum_{\nu/\mu}(-1)^{l(\nu/\mu)}\sqrt{J_\mu J_\nu}\mathcal{I}^\nu$$
where $\nu/\mu$ is a short rook strip.
The quantum Chevalley formula \cite[Thm. 3.9]{buch.chaput.ea:qkchev}, includes this as its first term, alongside the quantum correction. 
$$\mathcal{O}^{\mu}\star (1-\mathcal{O}^{s_\gamma})=\mathcal{O}^\mu\cdot (1-\mathcal{O}^{s_\gamma})-q\psi(\mathcal{O}^\mu\cdot (1-\mathcal{O}^{s_\gamma}))$$
$$=\sum_{\nu/\mu}(-1)^{l(\nu/\mu)}\sqrt{J_\mu J_\nu}\mathcal{I}^\nu-\sum_{\nu/\mu}(-1)^{l(\nu/\mu)}\sqrt{J_\mu J_\nu}q\psi\mathcal{I}^\nu$$
$$=\sum_{\nu/\mu}(-1)^{l(\nu/\mu)}\sqrt{J_\mu J_\nu}(\mathcal{I}^\nu -q\psi\mathcal{I}^\nu)$$
where we use linearity to move the $q$ and $\psi$ to be inside the second summation. 
Thus we have the following corollary. 

\begin{cor}\label{cor:cor2} Let $\mu\in W^P$, then
     $$\mathcal{O}^{\mu}\star (1-\mathcal{O}^{s_\gamma})=\sum_{\nu/\mu \text{ short rook strip}}(-1)^{l(\nu/\mu)}\sqrt{J_\mu J_\nu}\mathcal{I}_q^{\nu}$$
\end{cor}

This corollary gives a new way to express the Chevalley formula in $QK_T(X)$ in terms of the quantized ideal sheaves. 

\section{An Application}\label{sec:application}
Let $X=Gr(k,n)$, the Type A Grassmannian, for the remainder of this section. 
From \Theorem{quantum} we have that for $\lambda\in W^P$, $\mathcal{I}_q^\lambda=\frac{1}{J_\lambda}\mathcal{O}^\lambda\star (1-\mathcal{O}^{s_\gamma})$. 
Also, from \cite[Sec. 4]{buch.chaput.ea:qkchev} we have a $T$-equivariant exact sequence 
\begin{equation}\label{E:Lvarpi} 0\to \mathcal{L}^{\vee}_{\varpi_k} \otimes\mathbb{C}_{-\omega_\gamma}\to \mathcal{O}_X\to \mathcal{O}_{X^{s_\gamma}}\to 0\end{equation}
where $\mathcal{L}^{\vee}_{\varpi_k}=\det\mathcal{S}$ is the top exterior power of the tautological bundle on $X$, and thus $(1-\mathcal{O}^{s_\gamma})=\det\mathcal{S}\otimes \mathbb{C}_{-\omega_\gamma}$. 
Let $\mathcal{Q}=\mathbb{C}^n/\mathcal{S}$ be the quotient bundle. 
Then from \cite[Thm. 6.1]{gu.mihalcea.sharpe.zou:quantum} we can write $$\wedge^i\mathcal{S}\star\det\mathcal{Q}=(1-q)\wedge^{k-i}\mathcal{S}^{*}\cdot \det(\mathbb{C}^n).$$
We can then let $i=k$ to give the top exterior power of $\mathcal{S}$ on the left hand side, and $\wedge^0\mathcal{S}^{*}=[\mathcal{O}_X]=1\in QK_T(X)$ on the right hand side, to get that 
$$\det\mathcal{S}\star\det\mathcal{Q}=(1-q)\cdot\det(\mathbb{C}^n).$$
Using this equality, the definition of the quantum product in terms of the quantum $K$-metric, and \Theorem{quantum}, we can calculate the product in the following theorem. 
A similar argument should be possible in other cominuscule cases beyond the Type A Grassmannian, but there are not yet known formulae for $\det\mathcal{S}\star\det\mathcal{Q}$ in these other cases. 
There are conjectured formulas for $\det\mathcal{S}\star\det\mathcal{Q}$ in the Type C Lagrangian Grassmannian within \cite{gu.mihalcea.sharpe.zou:symplectic}, which have currently not been mathematically proven.

\begin{thm}\label{thm:application} 
Let $\mu\in W^P$. Then 
$$\det\mathcal{Q}\star \mathcal{O}^\mu=\sum_{\lambda\in W^P} [\mathbb{C}_{-\lambda.\omega_\gamma}]\det(\mathbb{C}^n)q^{d(\mu,\lambda)}\mathcal{O}^{\lambda}$$
where $d(\mu,\lambda)=dist_X(\mu,\lambda)$ is the minimal degree $d$ for which $q^d$ occurs in the product $X^\mu \star X_\lambda$ in the quantum cohomology ring of $X$.

\end{thm}

Note that if we reduce to the nonequivariant quantum $K$-theory ring, $QK(X)$, the coefficient for $\mathcal{O}^{\lambda}$ is simply $q^{d(\mu,\lambda)}$. 
This is coming from the formula for the pairing, $\left(\!\left(\mathcal{O}^u,\mathcal{O}_v\right)\!\right)=\frac{q^{d(u,v)}}{1-q}$, from \cite{buch.chaput.ea:euler} as discussed in \Section{Quantum1}. Also, nonequivariantly $X_\lambda=X^{(\lambda^{\vee})}$, and so $dist_X(\mu,\lambda)$ is the minimal degree $d$ for which $q^d$ occurs in the product $X^{\mu}\star X^{(\lambda^{\vee})}$ in $QH(X)$. 

\begin{proof}
Let $c_{\mu\lambda}$ be the structure constants for the above quantum product, i.e., let 
$$\det\mathcal{Q}\star \mathcal{O}^\mu=\sum_{\lambda\in W^P} c_{\mu\lambda}\mathcal{O}^{\lambda}.$$
Then we must have 
$$c_{\mu\lambda}=\left(\!\left(\det\mathcal{Q}\star\mathcal{O}^\mu,\mathcal{I}^q_\lambda\right)\!\right)$$
since $\mathcal{I}^q_\lambda=(\mathcal{O}_\lambda)^\vee\in QK_T(X)$ is the element dual to $\mathcal{O}^\lambda$. 
Recall that the left Weyl group action yields $w_0.\mathcal{I}^q_{\lambda}=\mathcal{I}_q^{(\lambda^{\vee})}$. For $X=Gr(k,n)$ we have $$\mathcal{I}_q^{(\lambda^{\vee})}=\frac{1}{J_{(\lambda^{\vee})}}\mathcal{O}^{(\lambda^{\vee})}\star \det\mathcal{S}\otimes \mathbb{C}_{-\omega_\gamma}=[\mathbb{C}_{-(\lambda^{\vee}).\omega_\gamma}]\mathcal{O}^{(\lambda^{\vee})}\star \det \mathcal{S}.$$
Thus from applying the action by $w_0$ to the structure constants we have 
$$w_0.c_{\mu\lambda}=w_0.\left(\!\left(\det \mathcal{Q}\star \mathcal{O}^\mu,\mathcal{I}_\lambda^q\right)\!\right)=\left(\!\left(w_0.(\det\mathcal{Q}\star\mathcal{O}^{\mu}),w_0.\mathcal{I}^q_\lambda\right)\!\right)$$
Also, $w_0.\det\mathcal{Q}=\det\mathcal{Q}$ since $\mathcal{Q}$ is a homogeneous bundle associated to a representation \cite[Lemma 5.4(b)]{gu.mihalcea.sharpe.zou:quantum}. Thus we have 

\[ \begin{split} w_0.c_{\mu\lambda} & =\left(\!\left(\det\mathcal{Q}\star\mathcal{O}_{(\mu^{\vee})},[\mathbb{C}_{-(\lambda^{\vee}).\omega_\gamma}]\mathcal{O}^{(\lambda^{\vee})}\star \det\mathcal{S}\right)\!\right) \\
& =[\mathbb{C}_{-(\lambda^{\vee}).\omega_\gamma}]\left(\!\left(\mathcal{O}^{(\lambda^{\vee})},\mathcal{O}_{(\mu^{\vee})}\star \det\mathcal{S}\star\det\mathcal{Q}\right)\!\right) \\
& =[\mathbb{C}_{-(\lambda^{\vee}).\omega_\gamma}](1-q)\det(\mathbb{C}^n)\left(\!\left(\mathcal{O}^{(\lambda^{\vee})},\mathcal{O}_{(\mu^{\vee})}\right)\!\right)\\
& =[\mathbb{C}_{-(\lambda^{\vee}).\omega_\gamma}]\det(\mathbb{C}^n)(1-q)\frac{q^{d(\lambda^\vee,\mu^\vee)}}{(1-q)} \\
& = [\mathbb{C}_{-(\lambda^{\vee}).\omega_\gamma}]\det(\mathbb{C}^n)q^{d(\lambda^\vee,\mu^\vee)}
\end{split}
\] 
where $d(\lambda^{\vee},\mu^\vee)=d(\mu,\lambda)$ because the minimal power of $q$ can be calculated nonequivariantly, and nonequivariantly $X^{(\lambda^{\vee})}\star X_{(\mu^{\vee})}=X^{\mu}\star X_{\lambda}$.
Therefore
$$c_{\mu\lambda}=w_0.w_0.c_{\mu\lambda}=w_0.[\mathbb{C}_{-(\lambda^{\vee}).\omega_\gamma}]\det(\mathbb{C}^n)q^{d(\mu,\lambda)}$$
$$c_{\mu\lambda}=[\mathbb{C}_{-w_0(\lambda^{\vee}).\omega_\gamma}]\det(\mathbb{C}^n)q^{d(\mu,\lambda)}=[\mathbb{C}_{-\lambda.\omega_\gamma}]\det(\mathbb{C}^n)q^{d(\mu,\lambda)}.$$
\end{proof}

We now proceed to give a simpler description of the equivariant coefficient $[\mathbb{C}_{-\lambda.\omega_\gamma}]\det(\mathbb{C}^n)$ from \Theorem{application}. 
Take $X=Gr(k,n)=GL(n)/P'$ where $P'$ is the corresponding maximal parabolic subgroup of $GL(n)$. 
We let $T'\subset GL(n)$ denote the maximal torus.
We wish to work with the characters of $T'$ as opposed to those of $T\subset SL(n)$ which were used for all calculations preceding this section. 
Let $\iota: SL(n)\hookrightarrow GL(n)$ denote the inclusion map. Then since $SL(n)/P = GL(n)/P'$, the functoriality of equivariant $K$-theory gives an induced map $id^{*}:K^{T'}(Gr(k,n))\rightarrow K^T(Gr(k,n))$.
The characters of $T'$ are given by $\diag(t_1,\ldots,t_n)\mapsto t_1^{m_1}\cdots t_n^{m_n}$ where $m_1,\ldots,m_n$ are integers. 
Thus the multiplicative group of characters of $T'$ can be identified with the additive group $\mathbb{Z}^n$. 

The roots of the Lie algebra of the torus $T$ are $\alpha_i=e_{i+1}-e_i$, for $1 \le i \le n-1$ \cite[p.64]{humphries:lie}, and they span a 
hyperplane in the Lie algebra of $T'$. The latter may 
be identified to $\mathbb{R}^n$, equipped with the standard inner product.
For $1 \le i \le n-1$, we identify the fundamental weight $\varpi_i$ with the vector $(1, \ldots ,1 , 0, \ldots, 0)$ with $i$ entries of $1$.
We also have another weight for $T'$, the determinant, $\varpi_n=(1,\ldots,1)$ with $n$ entries of $1$.  
The map $id^{*}$ must send the determinant to the identity, $\varpi_n\mapsto 1$, as $T$ is a subgroup of $SL(n)$. 
We then have that $id^{*}$ sends the vectors $(1,1, \ldots, 1, 0, \ldots, 0)$ to the corresponding fundamental weights of $T$. 
Furthermore, the restriction $id^{*}:K^{T'}(Gr(k,n))\rightarrow K^T(Gr(k,n))$ preserves the action of the Weyl group on the associated pairings of the root systems.

We take the partition associated to $\lambda\in W^P$, $(\lambda_1,\ldots \lambda_k)$, then the associated $w_\lambda \in S_n$ is the unique Grassmannian permutation (has a single descent at position $k$) that satisfies $w_\lambda(i) = i+ \lambda_{k-i+1}$, for $1 \le i \le k$. 
Let $I_\lambda:=\{w_\lambda(i) \ | \ 1 \le i\le k\}$, and let $J_\lambda$ be the complement $\{1,\ldots,n\} \setminus I_\lambda$.
Let $j_l$ for $1\le l\le (n-k)$ be the elements of $J_\lambda$ written in increasing order, and similarly define $i_m\in I_\lambda$ for $1\le m\le k$. The one-line notation for $w_\lambda$ is then given by $w_\lambda=i_1\cdots i_k j_1\cdots j_{n-k}$.
In order to find the weight $\lambda.\omega_\gamma$ we then simply take the natural action of $S^n$ on $\mathbb{Z}^n$ by permuting the entries of $\omega_\gamma=(1,\ldots,1,0,\ldots,0)=\varpi_k$ by $w_\lambda$. 
Since $\omega_\gamma$ has $1$ in its first $k$ entries, we then have that $w_\lambda$ permutes $\omega_\gamma$ to have $1$ in the $i_m$th entry for all $1 \leq m\leq k$, and $0$ elsewhere. 
Note that this combinatorial correspondence between $\lambda\in W^P$ and the weight $\lambda.\omega_\gamma$ depends on our convention of letting our Schubert classes be indexed by partitions whose length denote the codimension in $X$. 

Let $T_i=e^{m_i}$ denote characters corresponding to a basis for the Lie algebra of $T'$. 
Then $K^{T'}({\text{point}})=\mathbb{Z}[T_1^{\pm 1},\ldots,T_n^{\pm1}]$ is a Laurent polynomial ring in these characters. 
We then have that $\det(\mathbb{C}^n)=\prod_{i=1}^n T_i \in K^{T'}({\text{point}})$. 
Furthermore, $[\mathbb{C}_{\omega_\gamma}]=\prod_{i=1}^k T_i$, coming from, $(1,\ldots,1,0,\ldots,0)$, the $n$-tuple in $\mathbb{Z}^n$ associated to $\omega_\gamma$. 
Based on our description of $\lambda.\omega_\gamma$ in the preceding paragraph, we then have that $[\mathbb{C}_{\lambda.\omega_\gamma}]=\prod_{m=1}^k T_{i_m}$, where $w_\lambda=i_1\cdots i_k j_1\cdots j_{n-k}$ and $i_m\in I_\lambda$. 
Since $[\mathbb{C}_{-\lambda.\omega_\gamma}]$ is the multiplicative inverse of $[\mathbb{C}_{\lambda.\omega_\gamma}]$, 
we have that $$[\mathbb{C}_{-\lambda.\omega_\gamma}]\det(\mathbb{C}^n)=\frac{\prod_{a=1}^n T_a}{\prod_{m=1}^k T_{i_m}}=\prod_{l=1}^{n-k} T_{j_l}$$
where $j_l\in J_\lambda$. 

As a brief example, let $X=Gr(3,7)$ and take the partition $\mu=(3,2)$. 
We have that $w_\mu(1)=1+\mu_{3-1+1}=1+\mu_{3}=1$
. Similarly $w_\mu(2)=2+\mu_{3-2+1}=2+\mu_{2}=4$, and $w_\mu(3)=3+\mu_{3-3+1}=3+\mu_{1}=6$.
Thus $w_\mu=1462357$ and $w_\mu$ acts on $\omega_\gamma$ to give $\mu.\omega_\gamma=(1,0,0,1,0,1,0)$, with the entries of $1$ corresponding to the elements of $I_\mu$. 
Thus we have that $[\mathbb{C}_{\mu.\omega_\gamma}]=T_1T_4T_6$ and $[\mathbb{C}_{-\mu.\omega_\gamma}]\det(\mathbb{C}^n)=T_2T_3T_5T_7$. 
Here we can see the indices for the characters in $[\mathbb{C}_{-\mu.\omega_\gamma}]\det(\mathbb{C}^n)$ are given by the elements of  $J_\mu$. 

From the description of $[\mathbb{C}_{-\lambda.\omega_\gamma}]\det(\mathbb{C}^n)$ in terms of characters of $T'$ we have the following corollary. 

\begin{cor}\label{cor:detQ} Let $\mu\in W^P$. Then 
$$\det\mathcal{Q}\star \mathcal{O}^\mu= \sum_{\lambda\in W^P}(\prod_{l=1}^{n-k}T_{j_l})q^{d(\mu,\lambda)}\mathcal{O}^{\lambda}$$
    in $QK_T(Gr(k,n))$, where the $j_l$ correspond to the last $n-k$ entries of the one line notation for the Grassmannian permutation $w_\lambda$, and $d(\mu,\lambda)=dist_X(\mu,\lambda)$ is the minimal degree $d$ for which $q^d$ occurs in the product $X^\mu \star X_\lambda$ in the quantum cohomology ring of $X$.
\end{cor}

\begin{example}
A special case of \Theorem{application} is when we multiply $\det\mathcal{Q}$ by the identity in $QK_T(X)$, the Schubert class indexed by the empty partition $\mu=\varnothing$.
Since $X^{\varnothing}\star X_{\lambda}=X_\lambda$, there are no powers of $q$ appearing in this product, and thus $d(\varnothing,\lambda)=0$ for all $\lambda\in W^P$. 
Therefore $$\det\mathcal{Q}=\det\mathcal{Q}\star\mathcal{O}^{\varnothing}= \sum_{\lambda\in W^P}[\mathbb{C}_{-\lambda.\omega_\gamma}]\det(\mathbb{C}^n)\mathcal{O}^{\lambda}=\sum_{\lambda\in W^P}(\prod_{l=1}^{n-k}T_{j_l})\mathcal{O}^{\lambda}. $$  
\end{example}

We next provide another example of \Theorem{application}, which demonstrates a notable property of the multiplication by $\det \mathcal{Q}$. 
Positivity of $QK_T(X)$ \cite{buch.chaput.ea:positivity} tells us that general products of Schubert classes will result in alternating sums of Schubert classes, but this multiplication by $\det\mathcal{Q}$ results in coefficients that are truly positive in every case. 

\begin{example}
For the sake of simplicity and clarity we begin by working nonequivariantly, so we restrict $[\mathbb{C}_{-\lambda.\omega_\gamma}]\det(\mathbb{C}^n)\mapsto 1$ in \Theorem{application}. 
Also nonequivariantly, there is an isomorphism from $QH(X)\to QH(X)$ such that $X_\nu\mapsto X^{(\nu^{\vee})}$, allowing us to calculate $dist_X(\mu,\lambda)$ using only the opposite Schubert basis. 
Let $X=Gr(2,4)$, and let $\mu=(1)$. 
Then 
$$X^{(1)}\star X^{\varnothing}  =X^{(1)}$$
$$X^{(1)}\star X^{(1)}=X^{(1,1)}+X^{(2)}$$
$$X^{(1)}\star X^{(1,1)}=X^{(2,1)}$$
$$X^{(1)}\star X^{(2)}=X^{(2,1)}$$
$$X^{(1)}\star X^{(2,1)}=qX^{\varnothing}+X^{(2,2)}$$
$$X^{(1)}\star X^{(2,2)}=qX^{(1)}$$
Thus this last calculations gives $dist((1),(2,2)^{\vee})=dist((1),\varnothing)=1$ and we can see $dist((1),\lambda)=0$ for all other $\lambda$. 
Thus, nonequivariantly, \Theorem{application} tells us that 
$$\det\mathcal{Q}\star\mathcal{O}^{(1)}=\mathcal{O}^{(2,2)}+\mathcal{O}^{(2,1)}+\mathcal{O}^{(2)}+\mathcal{O}^{(1,1)}+\mathcal{O}^{(1)}+q\mathcal{O}^{\varnothing}.$$

To find the equivaraint coefficients, $[\mathbb{C}_{-\lambda.\omega_\gamma}]\det(\mathbb{C}^n)$, we write $w_\lambda$ for each $\lambda\in W^P$. 

For $\lambda=\varnothing$ we have $w_\lambda=1234$
and thus $[\mathbb{C}_{-\varnothing.\omega_\gamma}]\det(\mathbb{C}^n)=T_3T_4$.

For $\lambda=(1)$ we have $w_\lambda=1324$
and thus $[\mathbb{C}_{-(1).\omega_\gamma}]\det(\mathbb{C}^n)=T_2T_4$.

For $\lambda=(1,1)$ we have $w_\lambda=2314$
and thus $[\mathbb{C}_{-(1,1).\omega_\gamma}]\det(\mathbb{C}^n)=T_1T_4$.

For $\lambda=(2)$ we have $w_\lambda=1423$
and thus $[\mathbb{C}_{-(2).\omega_\gamma}]\det(\mathbb{C}^n)=T_2T_3$.

For $\lambda=(2,1)$ we have $w_\lambda=2413$
and thus $[\mathbb{C}_{-(2,1).\omega_\gamma}]\det(\mathbb{C}^n)=T_1T_3$.

For $\lambda=(2,2)$ we have $w_\lambda=3412$
and thus $[\mathbb{C}_{-(2,2).\omega_\gamma}]\det(\mathbb{C}^n)=T_1T_2$. 

We conclude from \Cor{detQ} that 

$$\det\mathcal{Q}\star\mathcal{O}^{(1)}=T_1T_2\mathcal{O}^{(2,2)}+T_1T_3\mathcal{O}^{(2,1)}+T_2T_3\mathcal{O}^{(2)}+T_1T_4\mathcal{O}^{(1,1)}+T_2T_4\mathcal{O}^{(1)}+T_3T_4q\mathcal{O}^{\varnothing}.$$

\end{example}

\providecommand{\bysame}{\leavevmode\hbox to3em{\hrulefill}\thinspace}
\providecommand{\MR}{\relax\ifhmode\unskip\space\fi MR }
\providecommand{\MRhref}[2]{%
  \href{http://www.ams.org/mathscinet-getitem?mr=#1}{#2}
}
\providecommand{\href}[2]{#2}

\end{document}